\documentclass[11pt, oneside,reqno]{amsart}
\usepackage[utf8]{inputenc}
\usepackage[english]{babel}
\usepackage[T1]{fontenc}
\usepackage{amsmath}
\usepackage{amsfonts}
\usepackage{amssymb}
\usepackage{amsthm}
\usepackage{amscd}
\usepackage{soul}
\usepackage{geometry}
\usepackage{enumitem} 
\usepackage{cancel} 
\usepackage{graphicx}
\usepackage{mathtools}
\usepackage{color}
\usepackage{comment}
\usepackage{mathdots}
\usepackage{algorithm}
\usepackage{algpseudocode}
\usepackage{tikz}
\usetikzlibrary{arrows}
\usepackage{color}
\usepackage{bbm}
\definecolor{marin}{rgb}   {0.,   0.3,   0.7} 
\definecolor{rouge}{rgb}   {0.8,   0.,   0.} 
\definecolor{sepia}{rgb}   {0.8,   0.5,   0.} 
\usepackage[colorlinks,citecolor=marin,linkcolor=rouge,
            bookmarksopen,
            bookmarksnumbered
           ]{hyperref}
\usepackage{tikz}
\usepackage{caption}
\usepackage{wrapfig}

\newcommand\N{\mathbb{N}}
\newcommand\T{\mathbb{T}}
\newcommand\Z{\mathbb{Z}}

\newcommand\R{\mathbb{R}}
\newcommand\C{\mathbb{C}}

\newcommand{\dd}{\, \mathrm{d}}

\newcommand{\enstq}[2]{\left\{#1~\middle|~#2\right\}}

\newcommand\eps{\varepsilon}
\renewcommand{\Re}{\operatorname{Re}}
\newcommand\loc{\mathrm{loc}}

\newcommand\Id{\mathrm{Id}}

\newcommand\Ran{\mathrm{Ran}}
\newcommand\Dir{\mathrm{Dir}}
\newcommand\Neu{\mathrm{Neu}}

\DeclareCaptionFormat{custom}
{%
    #1#2\textit{\small #3}
}
\DeclarePairedDelimiter\abs{\lvert}{\rvert}%
\DeclarePairedDelimiter\norm{\lVert}{\rVert}%
\newcommand{\tnorm}[1]{{\left\vert\kern-0.25ex\left\vert\kern-0.25ex\left\vert #1 
    \right\vert\kern-0.25ex\right\vert\kern-0.25ex\right\vert}}

\makeatletter
\let\oldabs\abs
\def\abs{\@ifstar{\oldabs}{\oldabs*}}
\let\oldnorm\norm
\def\norm{\@ifstar{\oldnorm}{\oldnorm*}}
\makeatother

\def\restriction#1#2{\mathchoice
              {\setbox1\hbox{${\displaystyle #1}_{\scriptstyle #2}$}
              \restrictionaux{#1}{#2}}
              {\setbox1\hbox{${\textstyle #1}_{\scriptstyle #2}$}
              \restrictionaux{#1}{#2}}
              {\setbox1\hbox{${\scriptstyle #1}_{\scriptscriptstyle #2}$}
              \restrictionaux{#1}{#2}}
              {\setbox1\hbox{${\scriptscriptstyle #1}_{\scriptscriptstyle #2}$}
              \restrictionaux{#1}{#2}}}
\def\restrictionaux#1#2{{#1\,\smash{\vrule height .8\ht1 depth .85\dp1}}_{\,#2}} 

\allowdisplaybreaks

\theoremstyle{plain}
\newtheorem{theorem}{Theorem} [section]

\newtheorem{lemma}[theorem]{Lemma}
\newtheorem{corollary}[theorem]{Corollary}
\newtheorem{proposition}[theorem]{Proposition}

\theoremstyle{remark}
\newtheorem{remark}[theorem]{Remark}

\title[Propagation of high regularity for logNLS]{On the propagation of high regularity for the logarithmic Schrödinger equation}

\author{Quentin Chauleur}
\address{Univ. Lille, CNRS, Inria, UMR 8524 - Laboratoire Paul Painlevé, F-59000 Lille, France. }
\email{quentin.chauleur@inria.fr}

\author{Guillaume Ferriere}
\address{Univ. Lille, CNRS, Inria, UMR 8524 - Laboratoire Paul Painlevé, F-59000 Lille, France. }
\email{guillaume.ferriere@inria.fr}


\begin{document}

\begin{abstract}
    We investigate both the instantaneous loss and the persistence of high regularity for the one-dimensional logarithmic Schrödinger equation in symmetric domains under various boundary conditions. We show that for a broad class of odd initial data, the~$H^s$-norm of solutions exhibits instantaneous blow-up for all~$s>\frac{7}{2}$. Conversely, we establish that~$H^3$-regularity is preserved for solutions that are odd with first-order cancellation, non-vanishing behavior away from the origin and Neumann boundary conditions on symmetric bounded domains. These theoretical results are further supported and illustrated by numerical simulations.
\end{abstract}

\maketitle

\section{Introduction}

We consider the one-dimensional logarithmic Schrödinger equation
\begin{equation} \label{logNLS} \tag{logNLS} 
  i\partial_t u + \Delta u= \lambda u \log |u|^2  
  \end{equation}
with initial condition $u(0)=\varphi$, time variable $t \in \R$ and space variable $x \in \Omega \subset \R$. We are interested both in the \textit{defocusing} case $\lambda >0$ and the \textit{focusing} case $\lambda <0$. Notably, this equation enjoys conservation of mass
\[ M(u(t)) \coloneqq \int_{\Omega} |u(t)|^2 = M(\varphi)  \]
which ensures the boundedness of the $L^2(\Omega)$-norm of the solution, as well as the conservation of energy 
\[ E(u(t)) \coloneqq \int_{\Omega} |\nabla u(t)|^2 + \lambda \int_{\Omega} |u(t)|^2 \log |u(t)|^2 = E(\varphi)  \] 
for all times $t \in \R$. \medskip

The logarithmic Schrödinger was first introduced by \textsc{Białynicki-Birula} and \textsc{Mycielski}~\cite{Birula1976} as a model for nonlinear wave
mechanics. Since then, it has found applications in various physical contexts, including nonlinear optics, nuclear physics, and quantum gravity. From a mathematical perspective, early studies were conducted by \textsc{Cazenave} and \textsc{haraux}~\cite{CazenaveHaraux1980}, followed by \cite{Cazenave1983}. This equation recently regained attention with the paper of \textsc{Guerrero}, \textsc{L\'{o}pez} and \textsc{Nieto}~\cite{Guerrero2010}, followed by the works of \textsc{Ardila}~\cite{Ardila2016}, of \textsc{d'Avenia}, \textsc{Montefusco}, \textsc{Squassina}~\cite{Avenia2014}, and of \textsc{Carles} and \textsc{Gallagher}~\cite{CarlesGallagher2018}. Over the past fifteen years, equation \eqref{logNLS} has been the subject of extensive study within the mathematical community, and many of its key properties are now well understood. For a broader overview of this topic, we refer the reader to the recent survey by \textsc{Carles}~\cite{Carles2022}. \medskip

What makes equation~\eqref{logNLS} particularly interesting is its remarkably different behavior compared to the more classical nonlinear Schrödinger equation with power-type nonlinearity~$\lambda |\psi|^{2\sigma} \psi$ for~$\sigma > 0$. Notably, equation~\eqref{logNLS} can be formally derived as the limit $\sigma \to 0$ of the power-type NLS, as discussed in \cite[Section 7]{Carles2022}. While power-type NLS equations typically exhibit \textit{scattering} in the defocusing case, meaning that solutions asymptotically behaves like a linear solution with a $t^{-\frac12}$ decay rate in one dimension, it was shown in \cite{CarlesGallagher2018} that solutions to \eqref{logNLS} decay even faster, at the rate $(t \sqrt{\lambda \log t})^{-\frac12}$,  which in fact rules out standard scattering behavior. On the other hand, the equation is globally well-posed in any dimension, and no finite-time blow-up can occur. Equation~\eqref{logNLS} also exhibits a form of scaling invariance, so that if $u$ is solution of \eqref{logNLS} with $u(0)=\varphi$,~$\kappa u e^{-2it\lambda \log \kappa}$ is also solution with initial condition $\kappa \varphi$ for any $\kappa >0$. This roughly means that the size of the initial data does not alter the dynamics of the solution, which is an uncommon feature among nonlinear dispersive PDEs. This property also induces that the flow map can't be $\mathcal{C}^1$ but at most Lipschitzean. Interestingly in the focusing case, nonlinear structures such as solitons and breathers are particularly well understood in the logarithmic setting, and the existence of multisolitons and multibreathers has been rigorously established in \cite{Ferriere2021,Ferriere2020}.  \medskip

A central theme in the study of nonlinear dispersive PDEs is the question of \textit{propagation of regularity}. For standard nonlinearities, classical energy estimates combined with Gronwall-type arguments typically guarantee that higher Sobolev regularities are propagated over time, see for instance \cite[Proposition 3.11]{Tao2006}. This means that regularity is neither destroyed nor spontaneously created (thanks to time-reversal symmetry), as long as the solution persists. This naturally leads to the question of the control of the growth of Sobolev norms, which is at the heart of the \textit{weak turbulence} theory, with seminal contributions by \textsc{Bourgain} \cite{Bourgain1996} and \textsc{Staffilani} \cite{Staffilani1997}. In the logarithmic case with defocusing nonlinearity, it is shown in \cite[Corollary 1.10]{CarlesGallagher2018} that 
\[  (\log t)^{\frac{s}{2}} \lesssim \| u(t) \|_{\dot{H}^s(\R)} \lesssim (\log t)^{\frac{s}{2}}  \]
for all $0 < s \leq 1$ and for initial data $\varphi \in H^1 \cap \enstq{f}{xf \in L^2}$. Thus, low-regularity Sobolev norms exhibit a slow, logarithmic growth in time. Note that well-posedness in low-regularity spaces $H^s$ for $0 \leq s \leq 1$ has recently been established in \cite{CarlesHayashiOzawa2024}. Finally, a straightforward argument based on Kato's trick \cite[Section 2.3]{CarlesGallagher2018}, replacing second-order spatial derivatives by first-order time derivative, can be used to propagate $H^2$ regularity. \medskip 

However, the propagation of $H^3$ regularity (and more generally, of higher Sobolev norms than $H^2$) has remained an open and challenging problem, as highlighted in \cite[Remark 2.5]{Carles2022} or \cite{CarlesHayashiOzawa2024}. This question stands as one of the last significant theoretical gaps in the analysis of the logarithmic Schrödinger equation. The difficulty lies in the limited regularity of the nonlinearity~$z \mapsto z \log |z|^2$, which introduces singularities upon spatial differentiation and prevents the direct application of standard techniques. \medskip

This article partially fills this gap, providing new insights into the persistence of regularity beyond the~$H^2$ threshold, by studying precisely the role of cancellation points in the dynamics of equation~\eqref{logNLS}. Our analysis reveals how such cancellation points critically affect the propagation of smoothness, and we present a concise argument demonstrating the instantaneous breakdown of $H^s$ regularity for $s>\frac{7}{2}$ under specific symmetry constraints. On the other hand, by precisely controlling such behavior, we also establish the first known instance of $H^3$-regularity preservation under similar assumptions.  \medskip

It is worth noting that away from vacuum, the logarithmic Schrödinger equation does not exhibit such pathological behavior: in~\cite{ChauleurFaou2022}, the authors show that solutions on the torus remain stable in high Sobolev norms over long times near plane waves. Interestingly in \cite{Birula1976}, the authors discovered that equation \eqref{logNLS} admits a rich family of smooth solutions on the whole space $\mathbb{R}^d$ with $d\geq1$, taking the form of explicit Gaussian profiles parametrized by solutions $a$ and $b$ of a set of coupled differential equations. These particular solutions write 
\[ u(t,x)=b(t)e^{-a(t) \frac{x^2}{2}}, \quad i \dot{a} - a^2 = \lambda \Re a, \quad i \dot{b} - ab = \lambda b \log |b|^2  \]
in the one dimensional setting $d=1$, and are discussed, for example, in \cite[Section 3]{Carles2022}. We also refer to~\cite{Ferriere2022}, where the equation is studied within a highly regular analytic framework.  \medskip

We further emphasize that the question of regularity propagation for equation~\eqref{logNLS} has been a recurring topic in numerical investigations, beginning with the early studies by \textsc{Bao}, \textsc{Carles}, \textsc{Su}, and \textsc{Tang} \cite{BaoCarlesSuTang2019,BaoCarlesSuTang2019bis,BaoCarlesSuTang2022} (see also \cite{AbidiGoubetMartin}). The lack of rigorous results at high spatial regularity has prevented the development of high-order numerical schemes, instead prompting interest in low-regularity methods such as in \cite{BaoMaWang2024}, or in approaches requiring enhanced time regularity~\cite{ParaschisZouraris2023}, a property that is far from guaranteed for the logarithmic Schrödinger equation. It is worth noting that these regularity issues also have significant implications in the presence of additional rough potentials, notably in the context of space-time multiplicative noise~\cite{BarbuRocknerZhang2017,CuiSun2023}, or space white noise potential as in~\cite{ChauleurMouzard2025}. \medskip

This paper is organized as follows. In Section~\ref{sec:main_results}, we state our main results and the corresponding functional setting. Section~\ref{sec:non_propagation} is devoted to the non-propagation of regularity for~$s>\frac{7}{2}$. Section~\ref{sec:_toy_model} introduces a linear toy model which captures the behavior of a first-order cancellation solution of equation~\eqref{logNLS}. Then, Section~\ref{sec:preliminiray_estimates} develops the necessary preliminary estimates used in the proof of Theorem  \ref{theo:propagation} in Section~\ref{sec:fixed_point}, which shows that~$H^3$-regularity is preserved for solutions that are odd with first-order cancellation and non-vanishing away from the origin on bounded domains. Finally, Section~\ref{sec:numerics} brings rigorous numerical simulations which illustrates and corroborates our theoretical results, and Section~\ref{sec:conclusion} summarizes our results and places them in a broader context, highlighting several related open questions. Note that the associated codes are available on the \textsc{Gitlab} page \url{https://plmlab.math.cnrs.fr/chauleur/codes/-/tree/main/propagation_regularity_logNLS}.

\section{Functional setting and main results} \label{sec:main_results}

\subsection{Notations}

Let $\Omega$ be an open subset of $\R$. We recall the standard Lebesgue spaces~$L^p(\Omega)$ and $L^{\infty}(\Omega)$ endowed with their usual norms
\[  \| f \|_{L^p(\Omega)} = \left( \int_{\Omega} |f(x)|^p \dd x  \right)^{\frac{1}{p}} \quad \text{and} \quad \| f \|_{L^{\infty}(\Omega)} = \sup_{x \in \Omega} |f(x)| \]
as well as the usual Sobolev spaces 
\[ H^k(\Omega) \coloneqq \enstq{f \in L^2(\Omega)}{ \partial_x^{\alpha} f \in L^2(\Omega) \ \text{for} \ \alpha\leq k} \]
for any integer $k \in \N$. If $\Omega=\R$, inhomogeneous and homogeneous Sobolev spaces can also be defined through the Fourier transform
\[ \widehat{f}(\xi)=\int_{\R} f(x) e^{-ix  \xi} \dd x  \]
respectively by the norms
\[  \| u \|_{H^k(\R)} = \left( \int_{\R} (1+ \xi^2)^{\frac{k}{2}} |\widehat{f}(\xi)|^2 \dd \xi  \right)^{\frac12} \quad \text{and} \quad \| u \|_{\dot{H}^k(\R)} = \left( \int_{\R} |\xi|^k |\widehat{f}(\xi)|^2 \dd \xi  \right)^{\frac12} \]
which naturally extends to the fractional Sobolev spaces $H^s(\R)$ and $\dot{H}^s(\R)$ for any real number $s \in \R$. We furthermore define $H^1_0(\Omega)$ as the closure of $\mathcal{C}^{\infty}_{c}(\Omega)$ in $H^1(\Omega)$. In particular if $\Omega$ is an open bounded subset of $\R$, we have
\[  H^1_0(\Omega) = \enstq{f \in H^1(\Omega)}{\forall x \in \partial \Omega, \ f(x)=0}. \]
Denoting the Japanese bracket $\langle x \rangle = \sqrt{1 + x^2}$, we also define the weighted spaces 
\begin{equation*}
    \mathcal{F} (H^\alpha (\Omega)) = \{ f \in L^2 (\Omega) \mid \langle x \rangle^\frac{\alpha}{2} f \in L^2 (\Omega) \}.
\end{equation*}
for any $\alpha \in \R$. For $2\pi$-periodic functions $u$, we introduce both Fourier coefficients of~$u$ (denoted $ \widehat{u}_n$ for $n \in \Z$) and periodic Sobolev spaces $H^s(\T)$ induced by the norm~$\| \cdot \|_{H^s(\T)}$ where
\[  \widehat{u}_n \coloneqq \frac{1}{2\pi} \int_{\T} u(x) e^{-in x} \dd x \quad \text{and} \quad \| u\|_{H^s(\T)}^2 \coloneqq \sum_{n \in \Z} (1+n^2)^s |\widehat{u}_n|^2. \]
Lastly, for $I \subset \R$ and interval and $X=X(\Omega)$ a Banach space, we recall Hölder spaces~$\mathcal{C}^{k,\alpha}(I; X)$ for any $k \in \N$ and $ 0 < \alpha \leq 1$, which consists of functions having continuous time derivatives up through order $k$ and such that the $k$-th partial time derivative is Hölder continuous with exponent $\alpha$, so that
\[ \| u \|_{\mathcal{C}^{k,\alpha}(I;X)} = \| u \|_{\mathcal{C}^k(I;X)} + \sup_{\substack{t,t'\in I \\ t \neq t'}} \frac{\|\partial_t^k f(t) - \partial_t^k f(t')\|_{X}}{|t-t'|^{\alpha}} < \infty.  \]
We may also use the convention $\mathcal{C}(\Omega)=\mathcal{C}^0(\Omega)$ and $\mathcal{C}(I;X)=\mathcal{C}^0(I;X)$, and denote by~$\mathcal{C}_b(\Omega)$ the space of bounded continuous functions, by $L^{\infty}_{\loc}(\Omega)$ the spaces of locally bounded functions, and by $\mathcal{C}_w ([0, T]; H^k (\Omega))$ the set of functions $u$ from $[0, T]$ with values in $H^k (\Omega)$ such that, for every $t_0 \in [0, T]$, there holds
\begin{equation*}
    u (t) \rightharpoonup u (t_0) 
    \qquad
    \text{in } H^k (\Omega)
    \text{ as } t \to t_0, t \in [0, T].
\end{equation*}

For the sake of clarity and conciseness, the norms $\|\cdot\|_{L^p}$, $\|\cdot\|_{H^s}$ and $\| \cdot \|_{\mathcal{C}^{k,\alpha}}$ and will be written without explicit reference to the domain $\Omega$ throughout the paper when the context will make it unambiguous. We also sometimes adopt the harmless slight abuse of notation $\nabla = \partial_x$ and~$\Delta=\partial_x^2$ in the one-dimensional setting.

\subsection{Background on logarithmic Schrödinger equations}

In the work of \textsc{Carles} and \textsc{Gallagher} \cite{CarlesGallagher2018}, the authors established that the Cauchy problem for \eqref{logNLS} is well posed for initial data in $H^1 (\R) \cap \mathcal{F} (H^\alpha (\R))$ for any $\alpha > 0$. Moreover, if the initial data belongs to $H^2 (\R)$, then the corresponding solution $u$ enjoys the same regularity, that is, $u \in L^\infty_{\textnormal{loc}} (\R; H^2 (\R))$.

More recently, \textsc{Hayashi} and \textsc{Ozawa} \cite{HayashiOzawa2025} refined this Cauchy theory for initial data in~$H^2$-based spaces both for $\R$ and for general domains with Dirichlet boundary conditions. In this setting, the domain of the Laplace operator is respectively defined by
\begin{equation} \label{eq:D_Lap_Dirichlet}
    D_{\R} (\Delta) \coloneqq H^2(\R) \quad \text{or} \quad D_{\Dir} (\Delta) \coloneqq H^2 (\Omega) \cap H^1_0 (\Omega).
\end{equation}
In the sequel, we adopt the convention $H^1_0 (\R) = H^1 (\R)$ if $\Omega = \R$, and generically denote by $D(\Delta)$  the domain of the Laplace operator on $\Omega$, which will be clear from context.

In \cite{HayashiOzawa2025} the authors also proved the existence and uniqueness of a solution to~\eqref{logNLS} for any initial datum in the space
\begin{equation} \label{eq:W_2_def}
    W_2 (\Omega) \coloneqq \{ \phi \in D (\Delta) \mid \phi \log{\abs{\phi}^2} \in L^2 (\Omega) \}.
\end{equation}
Moreover, they proved that $H^2$-regularity propagates on any bounded subdomain of~$\Omega$, and even globally when the equation is focusing ($\lambda < 0$). Their techniques can actually be generalized to Neumann or periodic boundary conditions, where the domain of the Laplace operator is respectively given by
\begin{equation} \label{eq:D_Lap_Neumann}
    D_{\Neu} (\Delta) \coloneqq \{ \phi \in H^2 (\Omega) \mid \restriction{\partial_n \phi}{\partial \Omega} = 0 \},
\end{equation}
or
\begin{equation} \label{eq:D_Lap_periodic}
    D_{\T} (\Delta) \coloneqq H^2 (\T).
\end{equation}
Let us point out that, when $\Omega$ is a bounded domain, there holds $W_2 (\Omega) = D (\Delta)$. Indeed, we know that
\begin{equation*}
    \abs{\phi \log{\abs{\phi}^2}} \lesssim_\delta \abs{\phi}^{1-\delta} + \abs{\phi}^{1+\delta}
\end{equation*}
for any $\delta \in (0, \frac{1}{2})$, allowing us to conclude using the embedding $L^2 (\Omega) \subset L^{2 - 2 \delta} (\Omega)$ on bounded domains, along with the Gagliardo-Nirenberg inequality for the second term.

We now gather all these Cauchy theories in the following lemma.

\begin{lemma} \label{lem:Cauchy}
    Let $\lambda \in \R \setminus \{ 0 \}$ and let $\Omega \subset \R$ be either:
    \begin{itemize}
        \item $\Omega = \R$, or an interval of $\R$ equipped with Dirichlet boundary conditions (so that $D (\Delta)$ is defined by \eqref{eq:D_Lap_Dirichlet}),
        \item an interval of $\R$ equipped with Neumann boundary conditions (so that $D (\Delta)$ is defined by \eqref{eq:D_Lap_Neumann}),
        \item $\Omega = \T$ with periodic boundary conditions (so that $D (\Delta)$ is defined by \eqref{eq:D_Lap_periodic}).
    \end{itemize}
    Then for any $\varphi \in W_2 (\Omega)$ (where $W_2 (\Omega)$ is defined in \eqref{eq:W_2_def}), there exists a unique solution
    \begin{gather*}
        u \in W^{1, \infty}_{\textnormal{loc}} (\R; L^2 (\Omega)), \quad
        \partial_t u \in \mathcal{C}_w (\R; L^2_{\textnormal{loc}} (\Omega)), \\
        \Delta u \in (\mathcal{C}_w \cap L^\infty_{\textnormal{loc}}) (\R; L^2_{\textnormal{loc}} (\Omega)),
    \end{gather*}
    to \eqref{logNLS} in the sense that
    \begin{equation} \label{eq:lognls_subdom}
        i\partial_t u + \Delta u= \lambda u \log |u|^2  \quad \text{in } L^2 (\omega)
    \end{equation}
    for all open bounded subset $\omega \subset \Omega$ and $a.e.$ $t \in \R$, with $u (0) = \varphi$. Moreover, if $\Omega$ is a bounded domain or if $\lambda > 0$, then $u \in \mathcal{C} (\R; W_2 (\Omega))$, and \eqref{eq:lognls_subdom} holds in $L^2 (\Omega)$ for all $t \in \R$.
\end{lemma}

Equation \eqref{logNLS} has also been studied on $\R$ under non-vanishing boundary conditions at infinity, commonly referred to in the mathematical literature as the Gross–Pitaevskii setting, where solutions satisfy $\abs{u} \to 1$ as $\abs{x} \to \infty$. In this context, the associated energy functional $\mathcal{E}_{\textnormal{logGP}}$ and the energy space $E_\textnormal{logGP}$ are crucial tools:

 \begin{equation*}
    \mathcal{E}_{\textnormal{logGP}} (u) \coloneqq \norm{\nabla
      u}_{L^2(\R)}^2 + \lambda \int_{\R} \Bigl( |u|^2\log|u|^2-|u|^2+1 \Bigr) \dd x,
\end{equation*}
\begin{equation*}
    E_\textnormal{logGP} \coloneqq \{ v \in H^1_\textnormal{loc} (\mathbb{R}) \, | \, \mathcal{E}_\textnormal{logGP} (v) < \infty \}.
\end{equation*} 

More precisely, the Cauchy theory has recently been developed by \textsc{Carles} and \textsc{Ferriere}~\cite{Carles_Ferriere_logGP} in the defocusing case.

\begin{theorem}[\cite{Carles_Ferriere_logGP}] \label{th:Cauchy_logGP}
    Let $\lambda > 0$. For any $\varphi \in E_\textnormal{logGP}$, there exists a unique weak solution~$u \in L^\infty_\textnormal{loc} (\mathbb{R}; E_\textnormal{logGP})$ to \eqref{logNLS}. Moreover, the flow of \eqref{logNLS} enjoys the following properties:
    \begin{itemize}
        \item The energy is conserved, namely $\mathcal{E}_{\textnormal{logGP}}
          (u(t))= \mathcal{E}_{\textnormal{logGP}}(\varphi)$ for all $t\in
          \R$.
        \item We have $u - \varphi\in
        \mathcal{C} (\mathbb{R}; L^2(\R))$.
        \item If $\Delta \varphi \in L^2$, then $u - \varphi \in W^{1, \infty}_\textnormal{loc}
        (\mathbb{R}; L^2(\R)) \cap L^\infty_\textnormal{loc} (\mathbb{R};
        H^2(\R))$. 
    \end{itemize}
\end{theorem}

We conclude this section with a classical result for solutions to nonlinear Schrödinger equations (namely, the preservation of the oddness property) which we include here for completeness.

\begin{lemma} \label{lemma_odd}
Let $\varphi \in W_2(\Omega)$ (resp. $\varphi \in E_{\text{logGP}}$ and $\partial_x \varphi \in H^1(\Omega)$) be an odd function, and let~$u$ be the unique solution to equation~\eqref{logNLS} provided by the Cauchy theory in Lemma~\ref{lem:Cauchy} (resp. by Theorem \ref{th:Cauchy_logGP}). Then $u(t)$ remains odd for all $t \in \mathbb{R}$.
\end{lemma}
\begin{proof}
    The proof is classical: if $(t,x) \mapsto u(t,x)$ is a solution to equation~\eqref{logNLS}, then $(t,x) \mapsto -u(t,-x)$ is also a solution with the same initial condition $\varphi$, since $\varphi(-x) = -\varphi(x)$ for all $x \in \Omega$. By uniqueness, it follows that~$u(t,x) = -u(t,-x)$ for all $x \in \Omega$, which gives the result.
\end{proof}

 \subsection{Main results}

We begin by considering the case of very high regularities. Due to the non-smoothness of the nonlinearity at zero, one expects a critical threshold beyond which regularity can no longer be propagated if the solution vanishes.
However, precisely identifying the loss of regularity is delicate, as the cancellation point may evolve in time or even be punctual in time, which may lead to a smaller loss of regularity, if any.

A natural way to track down such a cancellation is to impose symmetry on the solution. For instance, assuming the solution is odd ensures that $x = 0$ is a fixed cancellation point, allowing a focused analysis of the regularity near this point. To this end, we introduce the space of odd functions with non-vanishing derivative at the origin, namely
\begin{equation*}
    \mathfrak{D} (\Omega) = \{ \phi \in  H^1_{\textnormal{loc}} (\Omega) \mid \phi \text{ odd},\ \partial_x \phi (0) \neq 0 \}.
\end{equation*}

Our first result shows that, for initial data in $\mathfrak{D} (\Omega)$, the critical threshold for regularity propagation lies below $s = \frac{7}{2}$.

\begin{theorem} \label{th:non-prop_regularity}
    Let $\lambda \in \mathbb{R} \setminus \{ 0 \}$ and take domain $\Omega$ and initial data $\varphi$ such that either:
    \begin{itemize}
        \item $\Omega = \mathbb{R}$ and $\varphi \in W_2 (\R) \cap \mathfrak{D} (\R)$,
        \item $\Omega = (-a, a)$ for some $a > 0$ and $\varphi \in W_2 (\Omega) \cap \mathfrak{D} (\Omega)$ with adequate boundary conditions (such as Neumann, Dirichlet or periodic boundary conditions),
        \item $\Omega = \mathbb{R}$, $\lambda > 0$ and $\varphi \in E_\textnormal{logGP} \cap \mathfrak{D} (\R)$.
    \end{itemize}
    Let $s > \frac{7}{2}$ and assume furthermore that $\varphi$ satisfies $\varphi' \in H^{s-1} (\Omega)$. Let $u$ be the unique solution to \eqref{logNLS} given by the previous Cauchy theory of Lemma \ref{lem:Cauchy} or Theorem \ref{th:Cauchy_logGP} with initial data~$u(0)=\varphi$. Then, for all~$T > 0$ and~$x_0 > 0$, we have $u \notin L^\infty ((0, T); H^s ((-x_0, x_0)))$.
\end{theorem}

The proof, detailed in Section \ref{sec:non_propagation}, relies on several key steps.
First, we observe that the solution remains in~$\mathfrak{D} (\Omega)$, at least for  short times, by symmetry properties of the Schrödinger flow (in particular using Lemma \ref{lemma_odd}). Then, if the solution is also $H^s$ locally near the origin for a short time, all the terms in~\eqref{logNLS} are actually $\mathcal{C}^1$ close to the origin except at the precise cancellation point $x=0$.
This allows to differentiate \eqref{logNLS} with respect to space for any $x\neq0$ and to integrate it in time, yielding an expression for~$\partial_x u (t, x)$ for $x \neq 0$ in a neighborhood of the origin. Last, we show that this expression actually diverges like~$\log{\abs{x}}$ when~$x \to 0$, proving that~$\partial_x u (t)$ is not continuous, which is in contradiction with the assumption~$\partial_x u (t) \in H^{s-1}$ since $s-1 > \frac{5}{2}$.

\begin{remark}
   Regarding the assumptions in Theorem \ref{th:non-prop_regularity}, any boundary condition on~$\Omega = (-a, a)$ that ensures a well-posed Cauchy problem and preserves the oddness of solutions would also be suitable.
\end{remark}

\begin{remark}
    We believe that the additional assumption $\partial_x \phi (0) \neq 0$ is purely technical. It is expected that nonzero initial data $\varphi$ satisfying $\partial_x \varphi (0) = 0$ would in general lead to a solution $u$ such that $u (t) \in \mathfrak{D}$ in small times $t > 0$. See the discussion in Section~\ref{sec:conclusion} for further details.
\end{remark}

We now restrict our attention to the case of a bounded domain $\Omega=(-a , a )$ for some~$a>0$, and we add the Neumann boundary condition $\partial_x u (a)=\partial_x u (-a) = 0$ to equation~\eqref{logNLS}. Assuming $\varphi \in D_{\Neu}(\Delta)$ (where $D_{\Neu}(\Delta)$ is defined in \eqref{eq:D_Lap_Neumann}), we know from Lemma~\ref{lem:Cauchy} that there exists a unique solution $u \in \mathcal{C}(\R; D_{\Neu}(\Delta)) \cap \mathcal{C}^{0,\frac12}(\R;L^2(\Omega))$ of equation \eqref{logNLS} such that~$u(0)=\varphi$. We introduce the space of odd functions with first-order cancellation which do not vanish outside the origin
\[ \mathcal{D} \coloneqq \enstq{\phi \in D_{\Neu} (\Delta)}{\phi \ \text{odd}, \ \phi(x) \neq 0 \ \text{for all } x \in [-a, a] \setminus \{ 0 \}, \, \partial_x \phi (0) \neq 0}.  \]
We recall that $H^1 (\Omega) \subset \mathcal{C} ([-a, a]; \C)$, so that $\phi (\pm a)$ is well-defined and the second requirement in the definition of $\mathcal{D}$ makes sense.
Like for the previous space $\mathfrak{D}$, by continuity we get the existence of a time $T>0$ such that for all $t \in (-T,T)$, $u(t) \in \mathcal{D}$ if~$\varphi \in \mathcal{D}$ (using once again Lemma \ref{lemma_odd}). Note that, since $\mathcal{D} \subset \mathfrak{D} (\Omega)$, Theorem~\ref{th:non-prop_regularity} is also valid in this setting. However, to the best of the authors' knowledge, nothing is known concerning the (non-)propagation of $H^s$ regularity for $2 < s \leq \frac{7}{2}$.

Our second main result states that, at least for short times, the $H^3$ regularity is propagated if the initial data is in $\mathcal{D}$.

\begin{theorem} \label{theo:propagation}
Let $\varphi \in \mathcal{D} \cap H^3(\Omega) $ in $\Omega=(-a , a )$ for some~$a>0$. Let $u$ be the unique solution to \eqref{logNLS} given by the previous Cauchy theory of Lemma \ref{lem:Cauchy} with Neumann boundary conditions and initial data~$u(0)=\varphi$. Then there exists a time $T>0$ such that~$u(t) \in \mathcal{D} \cap H^3(\Omega) $ for all $t \in (-T,T)$.
\end{theorem}

The proof of this statement proceeds in three steps:
\begin{itemize}
	\item We begin by introducing in Section \ref{sec:_toy_model} a linear toy model that captures the local-in-time behavior of the system near cancellation points.
	\item Next, we establish preliminary estimates for products of functions in Sobolev spaces, along with properties and representations of odd functions in $\mathcal{D}\cap H^3(\Omega)$ in Section~\ref{sec:preliminiray_estimates}.
	\item Finally, we gather these arguments in Section \ref{sec:fixed_point} to carry out a detailed fixed point argument based on a derived formulation of the nonlinear equation \eqref{logNLS}.
\end{itemize}

Actually, we even prove a slightly better result in Proposition \ref{prop:main_result_propagation}, showing that the $H^3$ regularity can explode \textit{only if} $\inf_\Omega \abs{\mathcal{V} [\partial_x u (t)]} \to 0$, where
\begin{equation*}
    \mathcal{V} [f] (x) = \int_0^1 f(\sigma x) \dd \sigma, \qquad \forall x \in \Omega.
\end{equation*}
Furthermore, we believe that this characterization is not completely optimal, and that the~$H^3$ regularity might be propagated beyond this point. However, this might require further examinations and probably more refined techniques.

\begin{remark}
    A classical way to address the non-Lipschitz nature of the logarithmic nonlinearity is to introduce a small regularization parameter $\epsilon > 0$, and replace it by the saturated version
$z \mapsto z \log(|z|^2+\epsilon)$ which is Lipschitz continuous (see for instance \cite{CarlesGallagher2018}). Since the resulting equation is mass-subcritical, it is generally expected that high Sobolev regularity can be propagated, at least for short times, but obviously not uniformly in $\epsilon$ (in view of Theorem \ref{th:non-prop_regularity}). A natural strategy would then be to prove that, under the assumptions of Theorem~\ref{theo:propagation}, the propagation of $H^3$ regularity is uniform in $\epsilon$, thereby recovering the same result in the limit~$\epsilon \rightarrow 0$ for the original problem.

    However, this approach appears to be ineffective. The regularization destroys key algebraic properties of the logarithmic nonlinearity that are central to our analysis (see Remark~\ref{rem:not_regu_comp} in Section~\ref{sec:fixed_point} for more details). While one could attempt to take into account these algebraic features in the regularizing process, we find that the $H^3$ norm of the solution seems to diverge as~$\epsilon \to 0$. This behavior, although perhaps surprising at first, suggests that regularization may obscure the true dynamics of the logarithmic Schrödinger equation. Proper numerical simulations provide further evidence for this theoretical observation and reinforce the relevance of our approach based on the integrated quantity $\mathcal{V}[f]$.
\end{remark}

\section{Non-propagation of the $H^s$ regularity} \label{sec:non_propagation}

In this section, we prove Theorem \ref{th:non-prop_regularity} by contradiction, working under the assumptions on the domain $\Omega$ and the initial data $\varphi$ as stated in the theorem. The core idea is to show that, under the assumption of additional $H^s$ regularity on some space interval $I_{x_0} \coloneqq (- x_0, x_0)$, the function $(t,x) \mapsto \partial_x u (t,x)$ must coincide with an expression that diverges like $\log{\abs{x}}$ as $x \to 0$. To reach this conclusion, we begin by establishing a series of auxiliary lemmas describing the behavior of $u$ near the origin.

Prior to that, we provide the following result, whose proof is classical using interpolation between $L^2$ and $H^s$.

\begin{lemma} \label{lem:regularity}
    For any interval $I \subset \R$, for any $T > 0$ and for all $\theta \in (0, 1)$, there holds $\mathcal{C}^{0, 1-\theta} ([0, T]; H^{s \theta} (I)) \subset W^{1, \infty} ([0, T]; L^2 (I)) \cap L^\infty ((0, T); H^s (I))$.
\end{lemma}

\subsection{Behavior and integrability near the origin}

First, we precise the behavior of the solution near the cancellation point $x=0$ on $[0, T]$, possibly reducing $T$ if necessary. Note that we already know $u(t, 0) = 0$ for all $t \in \mathbb{R}$, since $u(t)$ is odd. As for the spatial derivative of $u$ at $x=0$, we demonstrate that it does not vanish at least within a small time interval. 

\begin{lemma} \label{lem:zeta_est}
    Let $\varphi$ satisfying the assumptions of Theorem \ref{th:non-prop_regularity} for some $s > \frac{7}{2}$ and $u$ be  the solution to \eqref{logNLS} with initial data $\varphi$.
    Then there exists $\tau_0 \in (0, T)$ and $\delta > 0$ such that for all $t \in [0, \tau_0]$, the function $\zeta (t) \coloneqq \partial_x u (t, 0)$ satisfies
    \begin{equation*}
        \abs{\zeta (t) - \zeta (0)} \leq \frac{\abs{\zeta (0)}}{2}, \qquad
        \abs{\zeta (t)} > \frac{\abs{\zeta (0)}}{2},
    \end{equation*}
    and
    \begin{equation} \label{eq:taylor_exp}
        \abs{ u(t, x) - \zeta (t) x } \leq \delta x^2.
    \end{equation}
    In particular, $\zeta (t) \neq 0$ for  $t \in [0, \tau_0]$.
\end{lemma}

\begin{proof}
    By Lemma \ref{lem:regularity} and usual Sobolev embedding $H^3 (I_{x_0}) \hookrightarrow \mathcal{C}^2_b (I_{x_0})$, we already have~$u \in \mathcal{C}^{0, 1-\frac{3}{s}} ([0, T]; \mathcal{C}^2_b (I_{x_0}))$. Thus, we infer $\partial_x u \in \mathcal{C}^0 ([0, T] \times I_{x_0})$ which in turn implies that~$\zeta \in \mathcal{C}^0([0, T])$. The conclusion for the first and second estimates follows as $\zeta (0) \neq 0$ by assumption.
    
    The last estimate \eqref{eq:taylor_exp} follows from Taylor's theorem along with the fact that $\partial^2_{x} u$ is uniformly bounded on $[0, T] \times I_{x_0}$.
\end{proof}

From such an expansion, it is then straightforward to see that near the origin, the function~$(t,x) \mapsto u(t,x)$ only vanishes at the point $x=0$.

\begin{corollary} \label{cor:u_comp_x}
   Taking the same assumptions and notations as in Lemma \ref{lem:zeta_est}, let~$x_1 \coloneqq \min (x_0, \frac{\abs{\zeta (0)}}{4 \delta}) > 0$.
    Then for all $t \in [0, \tau_0]$ and $x \in I_{x_1}$, there holds \eqref{eq:taylor_exp} as well as
    \begin{equation} \label{eq:taylor_in_1st_order}
        \frac{\abs{\zeta (0)}}{4} \abs{x} \leq \abs{u (t, x)} \leq (\abs{\zeta (t)} + \frac{\abs{\zeta (0)}}{4}) \abs{x}.
    \end{equation}
    In particular, for any $t \in [0, \tau_0]$ and $x \in I_{x_1}$, we have $u (t, x) = 0$ if and only if $x = 0$.
\end{corollary}

We now show that $\log{\abs{u(t, x)}^2}$ is integrable near the origin, and in fact belongs to $L^p$ spaces for all $p < \infty$.

\begin{lemma} \label{lem:int_ln_u}
    With the same assumptions and notations as in Lemma \ref{lem:zeta_est} and Corollary~\ref{cor:u_comp_x}, for all $p \in [1, \infty)$, there exists $K_p > 0$ such that, for all $t \in [0, \tau_0]$,
    \begin{equation*}
        \norm{\log{\abs{u (t)}^2}}_{L^p (I_{x_1})} \leq K_p.
    \end{equation*}
\end{lemma}

\begin{proof}
    For all $t \in [0, \tau_0]$ and $x \in I_{x_1} \setminus \{ 0 \}$, we can decompose the logarithm as
    \begin{equation*}
        \log{\abs{u(t,x)}^2} = 2 \log{\abs{\frac{u(t,x)}{x}}} - 2 \log{\abs{x}}.
    \end{equation*}
    We then observe that the function $\abs{\frac{u(t,x)}{x}}$ is uniformly bounded and uniformly bounded away from $0$ thanks to equation \eqref{eq:taylor_in_1st_order}, which implies that the first term is actually in~$L^\infty (I_{x_1})$ and thus in any $L^p (I_{x_1})$ with a uniform bound. The conclusion follows from the fact that~$\log{\abs{x}} \in L^p (I_{x_1})$ for all $p \in [1, \infty)$.
\end{proof}

We can also show that the non-linearity in \eqref{logNLS} is actually $H^1$ near $0$.

\begin{lemma} \label{lem:int_d_x_nonlin}
    With the same assumptions and notations as in Lemma \ref{lem:zeta_est} and Corollary~\ref{cor:u_comp_x}, for all $t \in [0, \tau_0]$, the function $u (t) \log{\abs{u (t)}^2} \in \mathcal{C}^1 (I_{x_1} \setminus \{ 0 \})$, and for all $x \in I_{x_1} \setminus \{ 0 \}$, 
    \begin{equation} \label{eq:d_x_nonlin}
        \partial_x \Bigl( u (t, x) \log{\abs{u (t, x)}^2} \Bigr) = \partial_x u (t,x) \log{\abs{u (t,x)}^2} + 2 \frac{u (t,x) \Re ( u(t,x) \overline{\partial_x u (t,x)})}{\abs{u (t,x)}^2}.
    \end{equation}
    Moreover, we have $u (t) \log{\abs{u (t)}^2} \in H^1 (I_{x_1})$ with a uniform bound with respect to $t \in [0, \tau_0]$.
\end{lemma}

\begin{proof}
    As we already know that the $u \in \mathcal{C}^0 ([0, T] \times I_{x_1})$, it suffices to show that we have~$\partial_x \Bigl( u (t) \log{\abs{u (t)}^2} \Bigr) \in \mathcal{C}^0 (I_{x_1} \setminus \{ 0 \}) \cap L^2 (I_{x_1})$. Given that $u (t, x) \in \mathcal{C}^1 (I_{x_1})$ for all $t \in \mathbb{R}$ and only vanishes at $x=0$, this ensures the desired regularity and justifies differentiating the nonlinearity, leading to equation~\eqref{eq:d_x_nonlin}.
    For the second term of the right-hand side in \eqref{eq:d_x_nonlin}, we can directly estimate that
    \begin{equation*}
        \abs{\frac{u (t,x) \Re ( u(t,x) \overline{\partial_x u (t,x)})}{\abs{u (t,x)}^2}} \leq \abs{\partial_x u (t,x)},
    \end{equation*}
    which proves that the quantity of interest is in $L^2 (I_{x_1})$. To handle the first term on the right-hand side of \eqref{eq:d_x_nonlin}, we use the continuity of $\partial_x u$ on $[0, T] \times I_{x_1}$ along with Lemma \ref{lem:int_ln_u} which implies a uniform bound in $L^2 (I_{x_1})$.
\end{proof}

\subsection{Integrated equation on the derivative and conclusion}

We first derive an evolution equation satisfied by the spatial derivative $\partial_x u$ by differentiating \eqref{logNLS} with respect to $x$, then we integrate this equation in time. This integral form will be instrumental in analyzing the local behavior of the derivative near the singularity at $x=0$.

\begin{lemma} \label{lem:kdv}
    With the same assumptions and notations as in Lemma \ref{lem:zeta_est} and Corollary~\ref{cor:u_comp_x}, there holds $\partial_x u \in W^{1, \infty} ((0, \tau_0); L^2 (I_{x_1}))$ and 
    for all $t \in [0, \tau_0]$ and $x \in I_{x_1} \setminus \{ 0 \}$,
    \begin{equation} \label{eq:d_x_logNLS}
        i \partial_t \partial_x u (t,x) + \frac{1}{2} \partial_{x}^3 u (t,x) = \lambda \partial_x u (t,x) \log{\abs{u (t,x)}^2} + 2 \lambda \frac{u (t,x) \Re ( u(t,x) \overline{\partial_x u (t,x)})}{\abs{u (t,x)}^2}.
    \end{equation}
    Therefore, %we have $\partial_x u \in \mathcal{C}^1 ((0, \tau_0); \mathcal{C}^0 (I_{x_1} \setminus \{ 0 \}))$, and
    for all $x \in I_{x_1} \setminus \{ 0 \}$, $\partial_x u ( \cdot, x) \in \mathcal{C}^1 ([0, \tau_0])$.
    Last, for all $t \in [0, \tau_0]$ and $x \in I_{x_1} \setminus \{ 0 \}$, there holds
    \begin{multline} \label{eq:int_form_dx_u}
        i \partial_x u (t,x) - i \varphi' (x) + \int_0^t \partial_{x}^3 u (\tau, x) \dd \tau \\ 
            = \lambda \int_0^t \partial_x u (\tau,x) \log{\abs{u (\tau,x)}^2} \dd \tau + 2 \lambda \int_0^t \frac{u (\tau,x) \Re ( u(\tau,x) \overline{\partial_x u (\tau,x)})}{\abs{u (\tau,x)}^2} \dd \tau.
    \end{multline}
\end{lemma}

\begin{proof}
    From Lemma \ref{lem:int_d_x_nonlin}, we know that the non-linear term $u \log |u|^2$  is $L^\infty ((0, \tau_0); H^1 (I_{x_1}))$, and so is $\partial_{x}^2 u$ by assumption. Thus, from the expression of equation \eqref{logNLS}, we directly get that $\partial_t u \in L^\infty ((0, \tau_0); H^1 (I_{x_1}))$, which in turn proves that $u \in W^{1, \infty} ((0, x_1), H^1 (I_{x_1}))$ thanks again to Lemma \ref{lem:regularity}. As a byproduct, we infer $\partial_x u \in W^{1, \infty} ((0, x_1); L^2 (I_{x_1}))$.

    Similar arguments also shows that all the terms appearing in equation \eqref{logNLS} are actually $\mathcal{C}^1 (I_{x_1} \setminus \{ 0 \})$ at $t \in [0, \tau_0]$ fixed, which allows to differentiate equation \eqref{logNLS} with respect to $x$ on $[0, \tau_0] \times (I_{x_1} \setminus \{ 0 \})$, leading to equation \eqref{eq:d_x_logNLS}.

    Moreover, by assumption on $u$ and from Lemma \ref{lem:regularity}, we have
    \begin{equation*}
        u \in \mathcal{C}^0 ([0, T]; H^{\tilde{s}} (I_{x_1})) \subset \mathcal{C}^0 ([0, T]; \mathcal{C}^3 (I_{x_1}))
    \end{equation*}
    for any $\tilde{s} \in (\frac{7}{2}, s)$, thus $\partial_x^3 u \in \mathcal{C} ([0, \tau_0] \times I_{x_1})$. By Lemma \ref{lem:int_d_x_nonlin}, the right-hand side of \eqref{eq:d_x_logNLS} belongs to $\mathcal{C} ([0, \tau_0] \times (I_{x_1} \setminus \{ 0 \}))$. This leads to $\partial_t \partial_x u (t,x) \in \mathcal{C} ([0, \tau_0] \times (I_{x_1} \setminus \{ 0 \}))$, and therefore $\partial_x u ( \cdot, x) \in \mathcal{C}^1 ([0, \tau_0])$ for all $x \in I_{x_1} \setminus \{ 0 \}$.

    As for the last equality \eqref{eq:int_form_dx_u}, for any $x \in I_{x_1} \setminus \{ 0 \}$ and $t \in [0, \tau_0]$, we can integrate equation \eqref{eq:d_x_logNLS} on $[0, t]$, which readily gives the integrated equation~\eqref{eq:int_form_dx_u}.
\end{proof}

We now isolate the divergent term in the integrated equation and show that the remaining terms are uniformly bounded. As a result, no compensation can occur, ultimately leading to the desired contradiction.

\begin{lemma} \label{lem:est_bound}
    With the same assumptions and notations as in Lemma \ref{lem:zeta_est} and Corollary~\ref{cor:u_comp_x}, there exists $K > 0$ such that, for all $t \in [0, \tau_0]$ and $x \in I_{x_1} \setminus \{ 0 \}$,
    \begin{equation*}
        \abs{\partial_x u (t,x)} + \abs{\varphi' (x)} + \abs{\int_0^t \partial_{x}^3 u (\tau, x) \dd \tau} + \abs{\int_0^t \frac{u (\tau,x) \Re ( u(\tau,x) \overline{\partial_x u (\tau,x)})}{\abs{u (\tau,x)}^2} \dd \tau} \leq K
    \end{equation*}
    and
    \begin{equation*}
        \abs{\int_0^t \partial_x u (\tau,x) \log{\abs{u (\tau,x)}^2} \dd \tau - 2 \log{x}\int_0^t \partial_x u (\tau,x) \dd \tau } \leq K.
    \end{equation*}
\end{lemma}

\begin{proof}
    In the first inequality, the first three terms are uniformly bounded by the fact that~$u \in \mathcal{C}^0 ([0, \tau_0]; \mathcal{C}_b^3 (I_{x_1}))$. As for the last term, we point out once again that it is well defined for all $t \in [0, \tau_0]$ and $x \in I_{x_1} \setminus \{ 0 \}$, as $u (\tau, x)$ does not vanish for any $\tau \in [0, \tau_0]$ and~$x \in I_{x_1} \setminus \{ 0 \}$. Using again the fact that for all $\tau \in [0, \tau_0]$ and $x \in I_{x_1} \setminus \{ 0 \}$,
    \begin{equation*}
        \abs{\frac{u (\tau,x) \Re ( u(\tau,x) \overline{\partial_x u (\tau,x)})}{\abs{u (\tau,x)}^2}} \leq \abs{\partial_x u (\tau,x)},
    \end{equation*}
    we readily get the boundedness of this last term by the fact that $\partial_x u \in \mathcal{C}^0 ([0, \tau_0]; \mathcal{C}_b^0 (I_{x_1}))$.
    
    For the second estimate of Lemma \ref{lem:est_bound}, we can directly write that
    \begin{equation*}
        \int_0^t \partial_x u (\tau,x) \log{\abs{u (\tau,x)}^2} \dd \tau - 2 \int_0^t \partial_x u (\tau,x) \dd \tau \log{x} = \int_0^t \partial_x u (\tau,x) \log{\abs{\frac{u (\tau,x)}{x}}^2} \dd \tau.
    \end{equation*}
    From Corollary \ref{cor:u_comp_x}, we know that
    \begin{equation*}
        \frac{\abs{\zeta (0)}}{4} \leq \abs{\frac{u (\tau,x)}{x}} \leq \abs{\zeta (t)} + \frac{\abs{\zeta (0)}}{4}.
    \end{equation*}
    and since $\zeta$ is bounded on $[0, \tau_0]$, this proves that $\log{\abs{\frac{u (\tau,x)}{x}}^2}$ is uniformly bounded on~$[0, \tau_0] \times (I_{x_1} \setminus \{ 0 \})$. On the other hand, the function $\partial_x u (\tau, x)$ is also bounded as it is continuous on $[0, \tau_0] \times I_{x_1}$, which ends the proof.
\end{proof}

\begin{lemma} \label{lem:est_ln_explode}
    With the same assumptions and notations as in Lemma \ref{lem:zeta_est} and Corollary~\ref{cor:u_comp_x}, for all $t \in [0, \tau_0]$,
    \begin{equation*}
        \abs{\int_0^t \partial_x u (\tau,x) \log{\abs{u (\tau,x)}^2} \dd \tau} \underset{x \rightarrow 0}{\longrightarrow} \infty.
    \end{equation*}
\end{lemma}

\begin{proof}
    By Lemma \ref{lem:est_bound}, we only have to prove that the same conclusion holds for the quantity $\abs{\log{x} \int_0^t \partial_x u (\tau,x) \dd \tau }$. First, we know that $\partial_x u$ is continuous on $[0, \tau_0] \times I_{\varepsilon_0}$, thus
    \begin{equation*}
        \int_0^t \partial_x u (\tau,x) \dd \tau \underset{x \rightarrow 0}{\longrightarrow} \int_0^t \partial_x u (\tau, 0) \dd \tau = \int_0^t \zeta (\tau) \dd \tau.
    \end{equation*}
    Moreover, from the estimates of Lemma \ref{lem:zeta_est}, we get
    \begin{align*}
        \abs{\int_0^t \zeta (\tau) \dd \tau} &= \abs{\int_0^t (\zeta (\tau) - \zeta (0)) \dd \tau + t \zeta (0)} \\
            &\geq \abs{\zeta (0)} t - \abs{\int_0^t (\zeta (\tau) - \zeta (0)) \dd \tau} \\
            &\geq \abs{\zeta (0)} t - \int_0^t \abs{\zeta (\tau) - \zeta (0)} \dd \tau \\
            &\geq \abs{\zeta (0)} t - \int_0^t \frac{\abs{\zeta (0)}}{2} \dd \tau \\
            &\geq \frac{\abs{\zeta (0)}}{2} t > 0.
    \end{align*}
    Since $\int_0^t \partial_x u (\tau,x) \dd \tau$ has a limit which is not $0$ as $x \rightarrow 0$, and as $x \mapsto \abs{\log{x}}$ diverges at infinity, we get the conclusion.
\end{proof}

We can now conclude the proof Theorem \ref{th:non-prop_regularity} by the announced argument of contradiction.

\begin{proof}[Proof of Theorem \ref{th:non-prop_regularity}]
    Let $s > \frac{7}{2}$, and assume that $\Omega$ and $\varphi$ satisfy the assumptions of Theorem \ref{th:non-prop_regularity}. We can then apply Lemma \ref{lem:kdv} with $\tau_0$ and $x_1$ respectively defined in Lemma \ref{lem:zeta_est} and Corollary \ref{cor:u_comp_x}, and get the integrated equation. However, from Lemma \ref{lem:est_bound}, all the terms in this equality are bounded except the quantity $\int_0^t \partial_x u (\tau,x) \log{\abs{u (\tau,x)}^2} \dd \tau$, which is unbounded by Lemma \ref{lem:est_ln_explode}, giving a contradiction.
\end{proof}

\section{Linear toy model} \label{sec:_toy_model}

The next three sections are dedicated to proving Theorem~\ref{theo:propagation}, and we work under its assumptions throughout. In particular, we recall that $\Omega = (-a, a)$. For convenience, we will assume that $a \geq 2$ in the following, but this stands as a superficial assumption as equation~\eqref{logNLS} enjoys spatial scaling invariance. 

In this section, we seek to understand how a first-order cancellation point influences the dynamics of the logarithmic Schrödinger equation. Formally, if~$u(x) \approx x \partial_x u(0)$ near the origin, substituting this ansatz into the logarithmic nonlinearity of \eqref{logNLS} suggests that the equation locally resembles a linear one governed by the self-adjoint operator $-\Delta +\lambda \log|x|^2$. The integrability of the logarithmic potential part then allows us to relate standard Sobolev norms to the ones induced by this operator. To rigorously justify this approach, we introduce a smooth cutoff function~$\chi$  to model the role of such cancellation point.

Let $\chi \in \mathcal{C}^{\infty}(\mathbb{R})$ such that $|\chi| \leq 1$, $\abs{\chi'} \leq 1$ and
\[ \chi(x)= \left\{
\begin{aligned}
    -1  \quad \text{if} \ x \leq -2, \\
    x \quad \text{if} \ |x| \leq \frac12,\\
    1  \quad \text{if} \ x \geq 2.
\end{aligned}
\right.  \]
Such a function $\chi$ is illustrated in Figure \ref{fig:function_chi}.

\input{tikz_picture.tex}

 We consider the linear equation with Dirichlet boundary conditions
\begin{equation} \label{eq:toy_model_sys}
    \begin{cases}
        i \partial_t v + \Delta v = \lambda v \log |\chi|^2, & (t,x) \in (\R,\Omega),\\
        v (t,x) =0, & t \in \mathbb{R},  x\in \partial \Omega, \\
        v(0)=\varphi,
    \end{cases}
\end{equation}
where we recall that $\lambda \in \R^*$. We also consider the associated operator and its domain
\[ A \coloneqq -\Delta +  \lambda\log |\chi|^2, \]
\begin{equation*}
    D(A) = H^2 (\Omega) \cap H^1_0 (\Omega).
\end{equation*}
The operator $A$ is clearly symmetric in $L^2(\Omega)$, as by integration by parts we get that
\[ \langle A v_1, v_2 \rangle = \langle v_1,Av_2 \rangle  \]
for all $v_1, v_2 \in H^2(\Omega) \cap H^1_0 (\Omega)$.
Lastly, we define, for some $\kappa >0$ yet to be fixed, the operator
\begin{gather*}
    A_\kappa \coloneqq A + \kappa, \\
    D(A_\kappa) = D(A),
\end{gather*}
which is symmetric as well. Before proving that these operators are self-adjoint, we first recall Gagliardo-Nirenberg inequality on bounded domains in one dimension, namely
\begin{equation} \label{eq_sobolev_inequality_omega} \tag{GN}
  \| \partial_x^j u \|_{L^{p}} \leq C_{\Omega} \left(  \| \partial_x^k u \|_{L^2}^{\theta} \|u \|_{L^2}^{1-\theta} + \|u \|_{L^2} \right) \leq  2 C_{\Omega} \norm{u}_{H^k}^{\theta} \norm{u}_{L^2}^{1-\theta}, 
\end{equation}
for a constant $C_{\Omega}>0$ which only depends on $\Omega$, and where 
\[ \frac{1}{p} + j + k \theta= \frac12 \quad \text{and} \quad \frac{j}{k} \leq \theta <1   \]
with $j,k \in \N$, $j \leq k$ and $p \geq 2$. We then show that these linear operators generates a Schrödinger flow.

\begin{lemma} \label{lem:self_adjointness}
 The operators $A$ and $A_{\kappa}$ are self-adjoint for all $\kappa \in\R$, and respectively generate evolution groups of unitary operators $t \mapsto e^{it A}$ and $t \mapsto e^{itA_{\kappa}}$, satisfying the identity $e^{itA_{\kappa}}=e^{it \kappa} e^{it A}$ for all $t \in \R$. Moreover:
 \begin{itemize}
     \item The application $t \mapsto e^{- i t A}$ is a group homomorphism.
     \item The evolution group $e^{-i t A}$ is well-defined in $H^1_0 (\Omega)$, in the sense that for all $\varphi \in H^1_0 (\Omega)$, system \eqref{eq:toy_model_sys} has a unique solution in $\mathcal{C}(\mathbb{R}, H^1_0 (\Omega))$ given by $e^{- i t A} \varphi$.
     
 \end{itemize} 
\end{lemma}
\begin{proof}
    We establish the result for the operator $A$, as the argument for $A_{\kappa}$ follows along identical lines. First, recall that the Laplace operator $\Delta$ equipped with Dirichlet boundary conditions is symmetric in $L^2(\Omega)$, and for any $f \in L^2(\Omega)$, the equation $(\mu \Id - i \Delta)u=f$ admits a unique solution $u \in D(A)$ for any $\mu>0$ by virtue of the Lax-Milgram theorem. Hence, the operator $-\Delta$ is self-adjoint since its range satisfies $\Ran(\mu \Id - i \Delta)=L^2(\Omega)$. 
    
    We now proceed along the lines of the proof of \cite[Theorem 5.5]{Pazy1983}. Define the potential term $V \coloneqq \lambda \log |\chi|^2$, which is real-valued and belongs to $ L^p(\Omega)$ for every $1 \leq p < \infty$. Applying the Cauchy–Schwarz inequality together with the Sobolev embedding \eqref{eq_sobolev_inequality_omega}, we infer that
    \[ \| Vu \|_{L^2} \leq \| u \|_{L^{4}} \| V \|_{L^4} \leq 2 C_{\Omega} |\lambda \| \log |\chi|^2 \|_{L^4} \| \Delta u \|_{L^2}^{\frac18} \| u \|_{L^2}^{\frac78}.  \] 
    From Young's inequality we then get that
    \[  \| Vu \|_{L^2} \leq \epsilon \|\Delta u \|_{L^2} + C(\epsilon) \|u \|_{L^2} \]
    for any $\epsilon>0$ and with $C(\epsilon)>0$. Therefore, by the perturbation result \cite[Theorem 3.3.2]{Pazy1983}, the operator $A$ is also self-adjoint. By Stone's theorem, it follows that $A$ generates a strongly continuous group of unitary operators $t \mapsto e^{itA}$ on $L^2(\Omega)$, which completes the proof.
\end{proof}

We now state an equivalence result between  $H^1$-norm and the norm induced by $A_{\kappa}$. 

\begin{lemma} \label{lemma_positivity_A}
    For $\kappa \geq 0$ if $\lambda <0$ or $\kappa=\kappa(\lambda,\Omega,\chi)>0 $ large enough if $\lambda >0$, the operator $A_{\kappa}$ is nonnegative in $L^2$, and there exists $c_1$, $C_1>0$ such that for all $v \in H^1_0 (\Omega) \cap H^2 (\Omega)$
    \[ 0 \leq c_1 \|v\|_{H^1}^2 \leq \langle A_{\kappa}v , v \rangle \leq C_1 \| v \|_{H^1}^2.  \]
    Last, the quadratic form $\langle A_{\kappa} v , v \rangle$ can be uniquely extended to $H^1_0 (\Omega)$.
\end{lemma}
\begin{proof}
    By a direct computation involving an integration by parts, we have for all $v \in H^1_0 (\Omega) \cap H^2 (\Omega)$ that
    \begin{equation} \label{eq_Au_u}
      \langle A_\kappa v, v \rangle  = \int_{\Omega} |\nabla v|^2 + \lambda \int_{\Omega} |v|^2 \log |\chi|^2 \dd x + \kappa \int_{\Omega} |v|^2.
    \end{equation}   
    We have by Hölder's inequality
    \begin{equation*}
       0 \leq - \int_{\Omega} |v|^2 \log |\chi|^2 \dd x \leq \| |v|^2 \|_{L^p} \| \log |\chi|^2 \|_{L^q}=\| v \|_{L^{2p}}^2 \| \log |\chi|^2 \|_{L^q}
    \end{equation*}  
    with $1=\frac{1}{p}+\frac{1}{q}$. We then use Gagliardo-Nirenberg inequality on bounded domains \eqref{eq_sobolev_inequality_omega}, so
    \begin{align*}
        0 \leq - \int_{\Omega} |v|^2 \log |\chi|^2 \dd x &\leq C_{\Omega}^2  \left( \| \nabla v \|_{L^2}^{\frac{1}{2q}} \|v \|_{L^2}^{\frac12 +\frac{1}{2p}} + \|v \|_{L^2} \right) \| \log |\chi|^2 \|_{L^q} \\
            &\leq 2 C_{\Omega}^2  \left( \| \nabla v \|_{L^2}^{\frac{1}{q}} \|v \|_{L^2}^{1 +\frac{1}{p}} + \|v \|_{L^2} \right)  \| \log |\chi|^2 \|_{L^q}.
    \end{align*}  
    
    If $\lambda <0$, we have $\lambda \log |\chi|^2 \geq 0$ and 
    \begin{equation*}
        \langle A_\kappa v, v \rangle \geq \min(1, \kappa) \|\nabla v\|_{L^2}^2 \geq C \min(1, \kappa) \|v\|_{H^1}^2.
    \end{equation*}
    Moreover, from the previous estimates, denoting $X=\norm{u}_{L^2}$ and $Y=\norm{\nabla u}_{L^2}$, there holds
    \begin{equation*}
        \langle A_\kappa v, v \rangle \leq Y^2 + 2|\lambda| C_{\Omega}^2 \| \log |\chi|^2 \|_{L^q} Y^{\frac{1}{q}} X^{1 +\frac{1}{p}} + \left( \kappa + 2 |\lambda| C_{\Omega}^2  \| \log |\chi|^2 \|_{L^q}\right) X^2,
    \end{equation*}
    so the result is straightforward. We now focus on the case $\lambda >0$. As $\lambda \int_{\Omega} |v|^2 \log |\chi|^2 \dd x \leq 0$, we get
    \begin{equation*}
        \langle A_\kappa v, v \rangle \leq \max(1, \kappa) \|v\|_{H^1}^2.
    \end{equation*}
    Moreover, with the previous estimates and with the same notations, we then get that
    \[   \langle A_\kappa v, v \rangle \geq Y^2 - 2|\lambda| C_{\Omega}^2 Y^{\frac{1}{q}} X^{1 +\frac{1}{p}}  + \left( \kappa -  2|\lambda| C_{\Omega}^2  \| \log |\chi|^2 \|_{L^q}\right) X^2 ,\]
    so taking $\kappa >0$ large enough (with respect to $\lambda$, $\Omega$ and $\chi$), we get the lower bound.
\end{proof}

% This result allows to define the square root $A_\kappa^\frac{1}{2}$ of the nonnegative symmetric operator $A_{\kappa}$ (see \cite[Proposition  2.81]{CheverryRaymond2021}), as well as the Schrödinger evolution group $t \mapsto e^{- i t A_{\kappa}}$ as well as $t \mapsto e^{- i t A}$, which will be needed for the study of the nonlinear problem. Moreover, we also have the following result concerning the domain of $A_\kappa^\frac{1}{2}$.

% \begin{corollary} \label{cor:domaine_sqrt_A}
%     There holds $D (A_\kappa^\frac{1}{2}) = H^1_0 (\Omega)$.
% \end{corollary}

% \textcolor{red}{A PROUVER OU AU MOINS METTRE UNE REF.}

From now on, the scalar $\kappa$ will be fixed so that the assumptions of Lemma \ref{lemma_positivity_A} hold. We also prove an equivalence result with the $H^2(\Omega)$-norm.

\begin{lemma} \label{lem:equiv_A_H2}
There exists $c_2$, $C_2>0$ such that for all $v \in H^2 (\Omega) \cap H^1_0 (\Omega)$
 \[ \norm{\Delta v}_{L^2} - c_2 \norm{v}_{H^1} \leq \norm{A_\kappa v}_{L^2} \leq C_2 \| v \|_{H^2}.  \]
\end{lemma}

\begin{proof}
    By triangular inequality, we get that
    \begin{equation*}
        \left| \norm{A_\kappa v}_{L^2} - \norm{\Delta v}_{L^2} \right| \leq \norm{A_\kappa v + \Delta v}_{L^2} = \norm{\lambda \log{\abs{\chi}^2} v + \kappa v}_{L^2}.
    \end{equation*}
    Similarly as in the previous proof, we get, for some $p, q$ such that $\frac{1}{p} + \frac{1}{q} = 1$,
    \begin{align*}
        \norm{\lambda \log{\abs{\chi}^2} v + \kappa v}_{L^2} &\leq \abs{\lambda} \norm{\log{\abs{\chi}^2} v}_{L^2} + \abs{\kappa} \norm{v}_{L^2} \\
            &\leq \abs{\lambda} \norm{\log{\abs{\chi}^2}}_{L^{2q}} \norm{v}_{L^{2p}} + \kappa \norm{v}_{L^2} \\
            &\leq c_2 (\lambda, \chi) \norm{v}_{H^1} + \kappa \norm{v}_{L^2},
    \end{align*}
    and the result follows.
\end{proof}

Now we turn our attention to the Schrödinger evolution group $e^{-i t A}$.

\begin{proposition} \label{prop:evol_group_schrod_A}
 The semi-group $e^{- i t A}$ commutes with $A$ and $A_\kappa$, and for all $\tau \in \mathbb{R}$ and $\varphi \in H^1_0 (\Omega)$, there holds
        \begin{gather*}
            \norm{e^{- i \tau A} \varphi}_{L^2} = \norm{\varphi}_{L^2}, \\
            \norm{e^{- i \tau A} \varphi}_{H^1} \leq \frac{\sqrt{C_1}}{\sqrt{c_1}} \norm{\varphi}_{H^1}.
        \end{gather*}
        If furthermore $\varphi \in H^2 (\Omega)$,
        \begin{equation*}
            \norm{e^{- i \tau A} \varphi}_{H^2} \leq C_2 \norm{\varphi}_{H^2} + (1 + c_2) \frac{\sqrt{C_1}}{\sqrt{c_1}} \norm{\varphi}_{H^1}.
        \end{equation*}
        Finally, for any $T > 0$ and $w \in \mathcal{C} ((-T,T), H^1_0 (\Omega))$, we have $e^{- i t A} (w (t)) \in \mathcal{C} ((-T,T), H^1_0 (\Omega))$. If moreover $w \in \mathcal{C} ((-T,T), H^1_0 (\Omega) \cap H^2 (\Omega))$, then $e^{- i t A} (w (t)) \in \mathcal{C} ((-T,T), H^1_0 (\Omega) \cap H^2 (\Omega))$.
\end{proposition}

\begin{proof}
    With Lemma \ref{lem:self_adjointness} and Lemma \ref{lemma_positivity_A}, the first properties follow directly from the general theory of Schrödinger evolution groups, see \cite{Pazy1983}. For the estimates on $e^{- i \tau A} \varphi$, we compute as follows:
    \[  \norm{e^{- i \tau A} \varphi}_{H^1} \leq \frac{1}{\sqrt{c_1}} \langle A e^{- i \tau A} \varphi,e^{- i \tau A} \varphi \rangle   =\frac{1}{\sqrt{c_1}} \langle A  \varphi, \varphi \rangle  \leq \frac{\sqrt{C_1}}{\sqrt{c_1}} \norm{\varphi}_{H^1},\]
    and
    \begin{align*}
        \norm{e^{- i \tau A} \varphi}_{H^2} &\leq \norm{A_\kappa e^{- i \tau A} \varphi}_{L^2} + (1 + c_2) \norm{e^{- i \tau A} \varphi}_{H^1} \\
            % &\leq \norm{A_\kappa e^{- i \tau A} \varphi}_{L^2} + (1 + c_2) \frac{\sqrt{C_1}}{\sqrt{c_1}} \norm{\varphi}_{H^1} \\
            &\leq \norm{A_\kappa \varphi}_{L^2} + (1 + c_2) \frac{\sqrt{C_1}}{\sqrt{c_1}} \norm{\varphi}_{H^1} \\
            &\leq C_2 \norm{\varphi}_{H^2} + (1 + c_2) \frac{\sqrt{C_1}}{\sqrt{c_1}} \norm{\varphi}_{H^1}.
    \end{align*}
    The proof for $w$ follows with the exact same lines taking the additional $\mathcal{C}(\left[-T,T\right])$-norm in time in account.
\end{proof}

\section{Preliminary estimates} \label{sec:preliminiray_estimates}

\subsection{Sobolev embeddings and inequalities}

We state some estimates regarding products, logarithms and inverse of functions in Sobolev spaces which will be useful in the sequel. We refer to \cite[Appendix B]{benzonigavage:hal-00693094} for the next two lemmas and their proofs.

\begin{lemma} \label{lem:prod_Sobolev}
    There exists $C > 0$ such that the following properties hold.
    \begin{itemize}
        \item For any $w_j \in H^1 (\Omega)$ for $1 \leq j \leq 4$, there holds
        \begin{gather*}
            \norm{w_1 w_2}_{H^1} \leq C \norm{w_1}_{H^1} \Bigl( \norm{w_2}_{L^\infty} + \norm{\partial_x w_2}_{L^2} \Bigr), \\
            \norm{w_1 w_2 - w_3 w_4}_{H^1} \leq C (\norm{w_2}_{H^1} \norm{w_1 - w_3}_{H^1} + \norm{w_3}_{H^1} \norm{w_2 - w_4}_{H^1}).
        \end{gather*}
        \item For any $w_j \in H^2 (\Omega)$ for $1 \leq j \leq 4$, there holds
        \begin{gather*}
            \norm{w_1 w_2}_{H^2} \leq C (\norm{w_1}_{H^2} \norm{w_2}_{L^\infty} + \norm{w_1}_{H^1} \norm{\partial_x w_2}_{H^1}), \\
            \norm{w_1 w_2 - w_3 w_4}_{H^2} \leq C \Bigl( \norm{w_2}_{H^2} \norm{w_1 - w_3}_{H^2} + \norm{w_3}_{H^2} \norm{w_2 - w_4}_{H^2} \Bigr).
        \end{gather*}
    \end{itemize}
\end{lemma}

\begin{proof}
    The first estimate in both items are classical and can be found in \cite{benzonigavage:hal-00693094}. As for the other two estimates, we decompose
    \begin{equation*}
        w_1 w_2 - w_3 w_4 = (w_1 - w_3) w_2 + w_3 (w_2 - w_4),
    \end{equation*}
    and the conclusion can be easily reached by using algebra properties of $H^1(\Omega)$ and $H^2(\Omega)$ in dimension one.
\end{proof}

Furthermore, we recall the following Moser estimates from \cite{benzonigavage:hal-00693094}:
\begin{lemma} \label{lem:comp_Sobolev}
    Let $I, J$ be two intervals of $\mathbb{R}$ with $\overline{J} \subset \mathring{I}$, and $v$ and $w$ two $J$-valued functions.
    \begin{itemize}
        \item If $F \in W^{1, \infty} (J)$, then
        \begin{equation*}
            \norm{F(w) - F(v)}_{L^2} \leq \norm{F'}_{L^\infty (J)} \norm{w - v}_{L^2}.
        \end{equation*}
        \item If $k \in \mathbb{N}$ and $F \in W^{k+1, \infty} (I)$, then
        \begin{equation*}
            \norm{\partial_x (F (v))}_{H^k} \leq C (1 + \norm{v}_{L^\infty})^{k} \norm{F'}_{W^{k, \infty} (I)} \norm{\partial_x v}_{H^k},
        \end{equation*}
        and, if $k > 0$,
        \begin{multline*}
            \norm{F(w) - F(v)}_{H^k} \leq C \norm{w - v}_{H^k} \\ \times \Bigl( (1 + \norm{w}_{L^\infty} + \norm{v}_{L^\infty})^{k-1} \norm{F'}_{W^{k, \infty} (I)} (\norm{w}_{H^k} + \norm{v}_{H^k}) + \norm{F'}_{L^\infty (J)} \Bigr),
        \end{multline*}
        for some $C$ depending only on $k, I, J$.
    \end{itemize}
\end{lemma}

We will frequently use estimates in Sobolev spaces for terms of the form $w_1 \log{\abs{w_2}^2}$ and~$\frac{w_1}{\abs{w_2}^2}$ for some functions $w_1$ and $w_2$ such that $w_2$ does not vanish. For this, we will use the following Sobolev estimates on $\log{\abs{w_2}^2}$ and $\frac{1}{\abs{w_2}^2}$, which follow directly from Lemmas \ref{lem:prod_Sobolev} and \ref{lem:comp_Sobolev}. We also introduce a lower bound $m$ to formulate these upcoming estimates. This lower bound will play a central role (appearing later as the constant $\alpha>0$) in controlling from below an integrated quantity related to the solution $u$ of equation \eqref{logNLS}.

\begin{lemma} \label{lem:special_Sobolev}
    Let $v \in H^1(\Omega)$ such that~$\inf \abs{v} \geq m$ for some $m > 0$. Then there holds
    \begin{gather*}
        \norm{\log{\abs{v}^2}}_{L^2} \leq \sqrt{2a} \Bigl( 2 \abs{\log{m^2}} + \abs{\log{\norm{v}_{L^\infty}^2}} \Bigr) \leq \sqrt{2a} \Bigl( \abs{\log{m^2}} + \abs{\log{\norm{v}_{H^1}^2}} + \log{C^2} \Bigr), \\
        \norm{\partial_x \Bigl( \log{\abs{v}^2} \Bigr)}_{L^2} \leq \frac{C}{m^2} \norm{v}_{H^1}^2, \\
        \norm{\frac{1}{\abs{v}^2}}_{L^2} \leq \frac{\sqrt{2a}}{m^2}, \\
        \norm{\partial_x \left( \frac{1}{\abs{v}^2} \right)}_{L^2} \leq \frac{C}{m^4} \norm{v}_{H^1}^2.
    \end{gather*}
    where $C>0$. Moreover, if $v \in H^2 (\Omega)$, then
    \begin{gather*}
        \norm{\partial_x \Bigl( \log{\abs{v}^2} \Bigr)}_{H^1} \leq \frac{C}{m^2} \Bigl( 1 + \frac{\norm{v}_{H^1}^2}{m^2} \Bigr) \norm{v}_{H^1} \norm{v}_{H^2}, \\
        \norm{\partial_x \left( \frac{1}{\abs{v}^2}\right)}_{H^1} \leq \frac{C}{m^4} \left( 1 + \frac{\norm{v}_{H^1}^2}{m^2} \right) \norm{v}_{H^1} \norm{v}_{H^2}.
    \end{gather*}
\end{lemma}

\begin{proof}
    Let $v_m = \frac{v}{m}$, so that $\abs{v_m} \geq 1$ on $\Omega$. Then $0 \leq \log{\abs{v_m}^2} \leq \log{\norm{v_m}_{L^\infty}^2}$ and we get
    \begin{equation*}
        \norm{\log{\abs{v_m}^2}}_{L^2} \leq \sqrt{2a} \norm{\log{\abs{v_m}^2}}_{L^\infty} \leq \sqrt{2a} \log{\norm{v_m}_{L^\infty}^2} = \sqrt{2a} (\log{\norm{v}_{L^\infty}^2} - \log{m^2}).
    \end{equation*}
    Therefore,
    \begin{align*}
        \norm{\log{\abs{v}^2}}_{L^2} = \norm{\log{\abs{v_m}^2} + \log{m^2}}_{L^2} &\leq \norm{\log{\abs{v_m}^2}}_{L^2} + \sqrt{2a} \abs{\log{m^2}} \\ &\leq \sqrt{2a} \Bigl( 2 \abs{\log{m^2}} + \abs{\log{\norm{v}_{L^\infty}^2}} \Bigr).
    \end{align*}
    On the other hand, by the continuous embedding of $L^\infty(\Omega)$ in $H^1 (\Omega)$, we have $\norm{v_m}_{L^\infty} \leq C \norm{v_m}_{H^1}$. Since $\norm{v_m}_{L^\infty} \geq 1$, we thus get
    \begin{equation*}
        \log{\norm{v_m}_{L^\infty}^2} \leq \log{C^2 \norm{v_m}_{H^1}^2} = \log{\norm{v_m}_{H^1}^2} + \log{C^2},
    \end{equation*}
    and the conclusion follows in the same way.

    For the second estimate, we see that $\partial_x \Bigl( \log{\abs{v}^2} \Bigr) = \partial_x \Bigl( \log{\abs{v_m}^2} \Bigr)$. Then, using Lemma~\ref{lem:comp_Sobolev} on $F(\abs{v_m}^2)$ with $F = \log$, $J = [1, \infty)$ and $I = [\frac{1}{2}, \infty)$, we get, for some $C > 0$,
    \begin{equation*}
        \norm{\partial_x \Bigl( \log{\abs{v_m}^2} \Bigr)}_{L^2} \leq C \norm{\partial_x \abs{v_m}^2}_{L^2},
    \end{equation*}
    but also, if $v \in H^2$,
    \begin{equation*}
        \norm{\partial_x \Bigl( \log{\abs{v_m}^2} \Bigr)}_{H^1} \leq C \Bigl( 1 + \norm{\abs{v_m}^2}_{L^\infty} \Bigr) \norm{\partial_x \abs{v_m}^2}_{H^1}.
    \end{equation*}
    With the definition of $v_m$ to come back to estimates on $v$, and by using Lemma \ref{lem:prod_Sobolev} on $\abs{v_m}^2 = v_m \overline{v_m}$, we get the result for both cases. The inequalities for $\frac{1}{\abs{v}^2}$ follow the same lines.
\end{proof}

We also have the following estimates concerning differences between such nonlinear terms:

\begin{lemma} \label{lem:diff_special_Sobolev}
    There exists $C > 0$ such that the following holds.
    Let $k \in \{ 1, 2 \}$ and $v, w \in H^k(\Omega)$ such that $\inf \abs{v} \geq m$ and $\inf \abs{w} \geq m$ for some $m > 0$, then
    \begin{multline*}
        \norm{\log{\abs{w}^2} - \log{\abs{v}^2}}_{H^k} \leq \frac{C}{m^2} \norm{w - v}_{H^k} (\norm{w}_{H^k} + \norm{v}_{H^k}) \\ \times \biggl( 1 + \Bigl(1 + \frac{\norm{w}_{H^1}^2 + \norm{v}_{H^1}^2}{m^2} \Bigr)^{k-1} \frac{\norm{w}_{H^1} \norm{w}_{H^k} + \norm{v}_{H^1} \norm{v}_{H^k}}{m^2} \biggl),
    \end{multline*}
    and
    \begin{multline*}
        \norm{\frac{1}{\abs{w}^2} - \frac{1}{\abs{v}^2}}_{H^k} \leq \frac{C}{m^4} \norm{w - v}_{H^k} (\norm{w}_{H^k} + \norm{v}_{H^k}) \\ \times \biggl( 1 + \Bigl(1 + \frac{\norm{w}_{H^1}^2 + \norm{v}_{H^1}^2}{m^2} \Bigr)^{k-1} \frac{\norm{w}_{H^1} \norm{w}_{H^k} + \norm{v}_{H^1} \norm{v}_{H^k}}{m^2} \biggl).
    \end{multline*}
\end{lemma}

\begin{proof}
    Similarly as in the previous proof, we introduce $v_m = \frac{v}{m}$ and $w_m = \frac{w}{m}$. The conclusion follows from applying Lemma \ref{lem:comp_Sobolev} to $\abs{v_m}^2$ and $\abs{w_m}^2$ with either $F = \log$ or $F (x) = \frac{1}{x}$ on the intervals $J = [1, \infty)$ and $I = [\frac{1}{2}, \infty)$, then using Lemma \ref{lem:prod_Sobolev} along with
    \begin{equation*}
        \abs{w_m}^2 - \abs{v_m}^2 = w_m (\overline{w_m - v_m}) + \overline{v_m} (w_m - v_m),
    \end{equation*}
    and lastly use the relation between $v_m, w_m$ and $v, w$ to come back onto estimates onto $v$ and $w$ only.
\end{proof}

\begin{lemma} \label{lem:computations_Sobolev}
    Let $m, R > 0$ and $k \in \{1, 2\}$. Then there exists $C = C_{m, R, k} > 0$ such that the following estimates hold: for any $w_j \in H^k (\Omega)$ such that $\inf{\abs{w_2}} \geq m$, $\inf{\abs{w_4}} \geq m$ and $\norm{w_j}_{H^k} \leq R$ for $1 \leq j \leq 4$, 
    \begin{gather*}
        \norm{\frac{w_1}{\abs{w_2}^2}}_{H^k} \leq C, \\
        \norm{\frac{w_1}{\abs{w_2}^2} - \frac{w_3}{\abs{w_4}^2}}_{H^k} \leq C (\norm{w_1 - w_3}_{H^k} + \norm{w_2 - w_4}_{H^k}), \\
        \norm{w_1 \log{\abs{w_2}^2}}_{H^k} \leq C, \\
        \norm{w_1 \log{\abs{w_2}^2} - w_3 \log{\abs{w_4}^2}}_{H^k} \leq C (\norm{w_1 - w_3}_{H^k} + \norm{w_2 - w_4}_{H^k}).
    \end{gather*}
\end{lemma}

\begin{proof}
    The first and third estimates are direct applications from Lemmas \ref{lem:prod_Sobolev} and \ref{lem:special_Sobolev}.
    For the other two estimates, we point out that
    \begin{gather*}
        \frac{w_1}{\abs{w_2}^2} - \frac{w_3}{\abs{w_4}^2} = \frac{1}{\abs{w_2}^2} (w_1 - w_3) + w_3 \Bigl( \frac{1}{\abs{w_2}^2} - \frac{1}{\abs{w_4}^2} \Bigr), \\
        w_1 \log{\abs{w_2}^2} - w_3 \log{\abs{w_4}^2} = \log{\abs{w_2}^2} (w_1 - w_3) + w_3 \Bigl( \log{\abs{w_2}^2} - \log{\abs{w_4}^2} \Bigr).
    \end{gather*}
    Using triangular inequality, the conclusion then comes from a straightforward application of Lemmas \ref{lem:prod_Sobolev} and \ref{lem:diff_special_Sobolev}.
\end{proof}

\subsection{A decomposition through Taylor formula}

We define the following linear application:
\begin{equation}
    \begin{array}{cccc}
        \mathcal{V}\;: & L^\infty (\Omega) & \longrightarrow & L^\infty (\Omega) \\
        \; & f & \longmapsto & \mathcal{V} [f] : (x \mapsto \int_0^1 f (\sigma x) \dd \sigma).
    \end{array}
\end{equation}
Once again, we begin with a formal discussion on the quantity at hand to provide the reader with the underlying intuition for its introduction. Instead of analyzing equation~\eqref{logNLS} directly, we follow an approach similar to that used in the proof of Theorem~\ref{th:non-prop_regularity} in Section~\ref{sec:non_propagation}, focusing instead on the spatial derivative~$\partial_x u$. This derivative satisfies a non-autonomous equation obtained by differentiating~\eqref{logNLS} with respect to space. To recast this equation in autonomous form, we replace~$u$ with the integrated quantity~$\mathcal{V}\left[ \partial_x u \right]$, thanks to the symmetric properties of~$u$. A precise understanding of how this integrated term behaves in Sobolev spaces is thus essential, and constitutes the main goal of this section.

It is direct to see that:
\begin{lemma} \label{lem:mcalV_Linfty}
    The application $\mathcal{V}$ is a continuous automorphism, with continuity constant $1$.
    It is also a continuous automorphism of $\mathcal{C} ([-a, a])$, with the same continuity constant.
    Moreover, if $f = f_0$ is constant, then $\mathcal{V} [f_0] = f_0$.
\end{lemma}

Going further, we develop the properties of $\mathcal{V}$ regarding Sobolev spaces.
\begin{lemma} \label{lem:V_sobolev}
We have:
    \begin{itemize}
        \item For any $f \in H^1 (\Omega)$, $\partial_x \mathcal{V} [f]$ and $\partial_{xx} \mathcal{V} [f]$ are well defined almost everywhere and there hold
        \begin{gather}
            \norm{\partial_x \mathcal{V} [f]}_{L^2} \leq \frac{\sqrt{2}}{2} \norm{\partial_x f}_{L^2}, \notag \\
            x \partial_x \mathcal{V} [f] = f - \mathcal{V} [f] \label{eq:rel_mcalVf_f}.% \\
            % x \partial_{xx} \mathcal{V} [f] = \partial_x f - 2 \partial_x \mathcal{V} [f]. \notag
        \end{gather}
        \item For any $f \in H^2 (\Omega)$, there also holds
        \begin{equation*}
            \norm{\partial_{xx} \mathcal{V} [f]}_{L^2} \leq \frac{1}{2} \norm{\partial_{xx} f}_{L^2}.
        \end{equation*}
        % $\partial_{xxx} \mathcal{V} [f]$ is also defined almost everywhere, and there also holds
        % \begin{gather*}
        %     \norm{\partial_{xx} \mathcal{V} [f]}_{L^2} \leq \frac{1}{2} \norm{\partial_{xx} f}_{L^2}, \\
        %     x \partial_{xxx} \mathcal{V} [f] = \partial_{xx} f - 3 \partial_{xx} \mathcal{V} [f].
        % \end{gather*}
    \end{itemize}

\end{lemma}

\begin{proof}
    Let $f \in \mathcal{C}^2 ([-a, a])$. Then it is easy to prove that $\mathcal{V} [f]$ is $\mathcal{C}^2$ on $\Omega$ and that, for all $x \in \Omega$,
    \begin{gather*}
        \partial_x \mathcal{V} [f] (x) = \int_0^1 \sigma \partial_x f (\sigma x) \dd \sigma, \\
        \partial_{xx} \mathcal{V} [f] (x) = \int_0^1 \sigma^2 \partial_{xx} f (\sigma x) \dd \sigma.
    \end{gather*}
    Thus, by Cauchy-Schwarz inequality, we get
    \begin{align*}
        \int_{-a}^a (\partial_x \mathcal{V} [f] (x))^2 \dd x &\leq \int_{-a}^a \int_0^1 \sigma^2 (\partial_x f (\sigma x))^2 \dd \sigma \dd x \\
            &\leq \int_0^1 \sigma^2 \int_{-a}^a (\partial_x f (\sigma x))^2 \dd x \dd \sigma \\
            &\leq \int_0^1 \sigma \int_{-\sigma a}^{\sigma a} (\partial_x f (y))^2 \dd y \dd \sigma \\
            &\leq \int_0^1 \sigma \norm{\partial_x f}_{L^2}^2 \dd \sigma = \frac{1}{2} \norm{\partial_x f}_{L^2}^2.
    \end{align*}
    Similarly, we have
    \begin{align*}
        \int_{-a}^a (\partial_{xx} \mathcal{V} [f] (x))^2 \dd x &\leq \int_{-a}^a \int_0^1 \sigma^4 (\partial_{xx} f (\sigma x))^2 \dd \sigma \dd x \\
            &\leq \int_0^1 \sigma^4 \int_{-a}^a (\partial_{xx} f (\sigma x))^2 \dd x \dd \sigma \\
            &\leq \int_0^1 \sigma^3 \int_{-\sigma a}^{\sigma a} (\partial_{xx} f (y))^2 \dd y \dd \sigma \\
            &\leq \int_0^1 \sigma^3 \norm{\partial_{xx} f}_{L^2}^2 \dd \sigma = \frac{1}{4} \norm{\partial_{xx} f}_{L^2}^2.
    \end{align*}
    
    Moreover, there holds
    \begin{align*}
        x \partial_{x} \mathcal{V} [f] (x) &= \int_0^1 \sigma \frac{\partial}{\partial \sigma} \Bigl( f (\sigma x) \Bigr) \dd \sigma \\
            &= f (x) - \int_0^1 f (\sigma x) \dd \sigma \\
            &= f (x) - \mathcal{V} [f] (x).
    \end{align*}
    % and also
    % \begin{align*}
    %     x \partial_{xx} \mathcal{V} [f] (x) &= \int_0^1 \sigma^2 \frac{\partial}{\partial \sigma} \Bigl( \partial_x f (\sigma x) \Bigr) \dd \sigma \\
    %         &= \partial_x f (x) - 2 \int_0^1 \sigma \partial_x f (\sigma x) \dd \sigma \\
    %         &= \partial_x f (x) - 2 \partial_x \mathcal{V} [f] (x).
    % \end{align*}
    % %
    % The second inequality is a direct consequence of this equality and of the previous $L^2$ inequality for $\partial_x \mathcal{V} [f]$. The equality for $x \partial_{xxx} \mathcal{V} [f]$ is reached in the same way.
    By density of~$\mathcal{C}^2 (\Omega)$ in $H^1 (\Omega)$ and in $H^2 (\Omega)$, we get the conclusion.
\end{proof}

From the equalities of the previous lemma, we can also derive some additional estimates.

\begin{corollary} \label{cor:xV_sobolev}
 For any $f \in H^1 (\Omega)$, there holds
        \begin{equation*}
            \norm{\mathcal{V} [f]}_{H^1} \leq C_{\Omega, a} \norm{f}_{H^1},
        \end{equation*}
        % \begin{gather*}
        %     \norm{\mathcal{V} [f]}_{H^1} \leq C_{\Omega, a} \norm{f}_{H^1}, \\
        %     \norm{x \partial_x \mathcal{V} [f]}_{L^2} \leq \norm{f}_{L^2} + 4a C_\Omega \norm{f}_{H^1}, \\
        %     \norm{x \partial_{xx} \mathcal{V} [f]}_{L^2} \leq (1 + \sqrt{2}) \norm{\partial_x f}_{L^2}, \\
        %     \norm{x \partial_x \mathcal{V} [f]}_{H^1} \leq \norm{f}_{L^2} + 4a C_\Omega \norm{f}_{H^1} + (1 + \frac{1}{\sqrt{2}}) \norm{\partial_x f}_{L^2} \leq K_{\Omega, a} \norm{f}_{H^1},
        % \end{gather*}
        for $C_{\Omega, a} \coloneqq \sqrt{2a C_\Omega + 2^{- \frac{1}{2}}}$. % and $K_{\Omega, a} \coloneqq 1 + 2^{-\frac{1}{2}} + 4 a C_\Omega$.
         Moreover, for any $f \in H^2 (\Omega)$, there holds
        \begin{equation*}
            \norm{\mathcal{V} [f]}_{H^2} \leq \left(C_{\Omega, a} + \frac{1}{2} \right) \norm{f}_{H^2}.
        \end{equation*}
        % \begin{gather*}
        %     \norm{\mathcal{V} [f]}_{H^2} \leq (C_{\Omega, a} + \frac{5}{2}) \norm{f}_{H^2}, \\
        %     \norm{x \partial_{xxx} \mathcal{V} [f]}_{L^2} \leq \frac{5}{2} \norm{\partial_{xx} f}_{L^2}, \\
        %     \norm{x \partial_x \mathcal{V} [f]}_{H^2} \leq K_{\Omega, a} \norm{f}_{H^1} + \frac{5}{2} \norm{\partial_{xx} f}_{L^2}.
        % \end{gather*}

\end{corollary}

We also provide an integral representation formula for functions vanishing at the origin, which will prove particularly useful in the following analysis.

\begin{lemma} \label{lemma_integral_u}
    Let $\phi \in H^2 (\Omega)$ such that $\phi (0) = 0$, then for all $x \in \Omega$,
    \begin{equation*}
        \phi(x) = x \, \mathcal{V} [\partial_x \phi] (x).
    \end{equation*}
    % \[ u(x)=x\int_0^1 \partial_x u(\sigma x) \dd \sigma.  \]
\end{lemma}
\begin{proof}
    This formula follows directly from the definition of $\mathcal{V}$ via a straightforward change of variables:
    \[ x \mathcal{V} \left[ \partial_x \phi \right](x)=x \int_0^1 \partial_x \phi(\sigma x) \dd \sigma = \int_0^x \partial_x\phi(y) \dd y = \phi(x)-\phi(0) = \phi(x)   \]
    as $\phi(0)=0$.
\end{proof}

This formula also provides a more compact and equivalent characterization of functions in $\mathcal{D}$, that we state and prove in the following lemma

\begin{lemma} \label{lem:equiv_mcalD}
    Let $\phi \in H^2 (\Omega)$ be an odd function. Then the assertion $\phi \in \mathcal{D}$ is equivalent to the following property:
    \begin{equation} \label{assertion:mcalV_bbb}
        \text{There exists } \alpha > 0 \text{ such that } \abs{\mathcal{V} [\partial_x \phi] (x)} \geq \alpha \text{ for all } x \in [-a, a].
    \end{equation}
\end{lemma}

\begin{proof}
    Let $\phi \in H^2 (\Omega)$ be an odd function and let us prove this equivalence.
    
    Assume that $\phi \in \mathcal{D}$. Since $\partial_x \phi \in H^1 (\Omega) \subset \mathcal{C} ([-a, a])$, Lemma \ref{lem:mcalV_Linfty} implies that $\mathcal{V} [\partial_x \phi] \in \mathcal{C} ([-a, a])$. Moreover, we have
    \begin{equation*}
        \mathcal{V} [\partial_x \phi] (0) = \partial_x \phi (0) \neq 0.
    \end{equation*}
    Therefore, we can find some $\delta > 0$ such that, for all $x \in (- \delta, \delta)$, there holds
    \begin{equation*}
        \abs{\mathcal{V} [\partial_x \phi] (x)} \geq \frac{1}{2} \abs{\partial_x \phi (0)} > 0.
    \end{equation*}
    On the other hand, $[-a , a] \setminus (- \delta, \delta)$ is compact, so that $\abs{\varphi}$ (which is continuous) reaches its minimum $\gamma$, which is not zero by assumption. Then, using the formula of Lemma \ref{lemma_integral_u}, we get that for all $x \in [-a , a] \setminus (- \delta, \delta)$,
    \begin{equation*}
        \abs{\mathcal{V} [\partial_x \phi] (x)} = \frac{\abs{\phi (x)}}{\abs{x}} \geq \frac{\gamma}{a},
    \end{equation*}
    leading to the property \eqref{assertion:mcalV_bbb} with $\alpha = \min{(\frac{1}{2} \abs{\partial_x \phi (0)}, \frac{\gamma}{a})}$.

    Conversely, if \eqref{assertion:mcalV_bbb} holds, it is direct to see that $\phi \in \mathcal{D}$ from Lemma \ref{lemma_integral_u} (showing that $\varphi$ cannot vanish except in $x = 0$) and from $\partial_x \phi (0) = \mathcal{V} [\partial_x \phi] (0)$.
\end{proof}

\begin{lemma} \label{lemma_minoration_dx_u}
Let $u \in \mathcal{C} ((-T,T); H^2 (\Omega))$ such that $\restriction{u}{t = 0} \in \mathcal{D}$ and for all $t \in (-T,T)$, the function $t\mapsto u(t)$ is odd. Let $v= \partial_x u$, then for $T_0>0$ small enough (in particular~$T_0 \leq T$), there exists a constant $\alpha>0$ such that, for all $t \in \left(- T_0,T_0\right)$, 
\[ \left| \mathcal{V} [v (t)] \right| \geq \alpha >0.  \]
\end{lemma}

\begin{proof}
    From Lemma \ref{lem:equiv_mcalD}, we know that there exists $\alpha > 0$ such that
    \begin{equation*}
        \abs{\mathcal{V} [\restriction{v}{t=0}]} \geq \alpha \qquad \text{on } \Omega.
    \end{equation*}
    On the other hand, from the regularity assumption on $u$ and Sobolev embedding, there holds~$u \in \mathcal{C} ([- T, T]; \mathcal{C}^1 ([-a, a]))$, so that there exists $T_0 \leq T$ such that, for all $t \in (-T_0, T_0)$,
    \begin{equation*}
        \norm{v (t) - \restriction{v}{t=0}}_{L^\infty} < \frac{\alpha}{2}.
    \end{equation*}
    Therefore, we also get for all $t \in (-T_0, T_0)$
    \begin{equation*}
        \abs{\mathcal{V} [v (t)]} \geq \abs{\mathcal{V} [\restriction{v}{t=0}]} - \abs{\mathcal{V} [v (t) - \restriction{v}{t=0}]} \geq \alpha - \norm{v (t) - \restriction{v}{t=0}} \geq \alpha - \frac{\alpha}{2} = \frac{\alpha}{2} > 0. \qedhere
    \end{equation*}
\end{proof}

 Finally,  we introduce the following function, which will be extensively used in the next section:
\begin{equation} \label{def:calX}
    \mathfrak{X} \coloneqq \frac{\chi}{x}.
\end{equation}
%
%\textcolor{blue}{CHANGER NOTATION}
A direct property satisfied by $\mathfrak{X}$ is the following:

\begin{lemma} \label{lem:prop_mcalX}
    The function $\mathfrak{X}$ is $\mathcal{C}^\infty (\mathbb{R})$ and satisfies $\mathfrak{X} \geq C_\chi$ on $\mathbb{R}$ for some constant $C_\chi > 0$. Moreover, for any $y > 2$
    \begin{equation*}
        y \frac{\mathfrak{X}' (y)}{\mathfrak{X} (y)} = - 1.
    \end{equation*}
\end{lemma}

\section{Fixed-point argument} \label{sec:fixed_point}

We turn to the proof of Theorem \ref{theo:propagation}, gathering properties from Section \ref{sec:_toy_model} on the linear toy model and from Section \ref{sec:preliminiray_estimates} on estimates on Sobolev spaces.
For some initial data $\varphi \in \mathcal{D} \cap H^3 (\Omega)$ where $\Omega = (-a, a)$, Lemma \ref{lem:Cauchy} gives a unique solution $u \in \mathcal{C} (\R, W_2 (\Omega))$ to \eqref{logNLS}. Moreover, from Lemma \ref{lemma_odd}, $u(t)$ is odd for all $t \in \R$.

We first rewrite equation~\eqref{logNLS} as 
\[ i \partial_t u = A u + \lambda u \log \left( \frac{|u|^2}{|\chi|^2} \right).  \]
Using Lemma \ref{lemma_integral_u} and $\mathfrak{X}$ defined in \eqref{def:calX}, we know that
\begin{equation} \label{eq:relation_u_V}
    \frac{u (t)}{\chi} = \frac{\mathcal{V} [\partial_x u (t)]}{\mathfrak{X}}.
\end{equation}
Thus, assuming that $\abs{\mathcal{V} [\partial_x u (t)]} \geq \alpha$ for some $\alpha > 0$ on $(- T, T)$ for some $T > 0$ given by Lemma \ref{lemma_minoration_dx_u}, we obtain the equation
\begin{equation*}
    i \partial_t u = A u + \lambda u \log{\abs{\mathcal{V} [\partial_x u (t)]}^2} - 2 \lambda u \log \mathfrak{X}.
\end{equation*}
Differentiating with respect to $x$, and denoting $v=\partial_x u$, we get
\begin{equation*}
    i \partial_t v = A v + 2 \lambda u \frac{\chi'}{\chi} + \lambda v \log{\abs{\mathcal{V} [v(t)]}^2} + \lambda u \partial_x \Bigl( \log{\abs{\mathcal{V} [v(t)]}^2} \Bigr) - 2 \lambda v \log \mathfrak{X} - 2 \lambda u \frac{\mathfrak{X}'}{\mathfrak{X}}.
\end{equation*}
Using once again Lemma \ref{lemma_integral_u} and equation \eqref{eq:relation_u_V}, and also Lemma \ref{lemma_integral_u} after expanding the fourth term, we get
\begin{multline*}
     i \partial_t v = A v + 2 \lambda \mathcal{V} [v(t)] \frac{\chi'}{\mathfrak{X}} + \lambda v \log{\abs{\mathcal{V} [v(t)]}^2} + 2 \lambda x \mathcal{V} [v(t)] \frac{\Re \Bigl( \partial_x \mathcal{V} [v(t)] \, \overline{\mathcal{V} [v(t)]} \Bigr)}{\abs{\mathcal{V} [v(t)]}^2} \\ - 2 \lambda v \log \mathfrak{X} - 2 \lambda x \mathcal{V} [v (t)] \frac{\mathfrak{X}'}{\mathfrak{X}}.
\end{multline*}
Now, we expand even more the fourth term with equality \eqref{eq:rel_mcalVf_f} from Lemma \ref{lem:V_sobolev}, which yields
\begin{multline*}
     i \partial_t v = A v + 2 \lambda \mathcal{V} [v(t)] \frac{\chi'}{\mathfrak{X}} + \lambda v \log{\abs{\mathcal{V} [v(t)]}^2} + 2 \lambda \mathcal{V} [v(t)] \frac{\Re \Bigl( v (t) \, \overline{\mathcal{V} [v(t)]} \Bigr)}{\abs{\mathcal{V} [v(t)]}^2} \\ - 2 \lambda \left(1 + x \frac{\mathfrak{X}'}{\mathfrak{X}} \right) \mathcal{V} [v(t)] - 2 \lambda v \log \mathfrak{X}.
\end{multline*}

We now define the following quantities
\begin{align*}
    V_1 [f] &\coloneqq 2 \lambda \mathcal{V} [f] \Bigl( \frac{\chi'}{\mathfrak{X}} - 1 -  x \frac{\mathfrak{X}'}{\mathfrak{X}} \Bigr), \\
    V_2 [f] &\coloneqq \lambda f \log{\abs{\mathcal{V} [f]}^2}, \\
    V_3 [f] &\coloneqq 2 \lambda \mathcal{V} [f] \frac{\Re \Bigl( f \, \overline{\mathcal{V} [f]} \Bigr)}{\abs{\mathcal{V} [f]}^2}, \\
    % V_4 &\coloneqq - 2 \lambda \mathcal{V} [v(t)] \frac{\Re \Bigl( \mathcal{V} [v(t)] \, \overline{\mathcal{V} [v(t)]} \Bigr)}{\abs{\mathcal{V} [v(t)]}^2}, \\
    V_4 [f] &\coloneqq - 2 \lambda f \log \mathfrak{X},%, \\
    % V_5 [f] &\coloneqq - \kappa f.
\end{align*}
so that the previous equation can be rewritten
\begin{equation} \label{eq:v}
    i \partial_t v = A v + V_1 [v] + V_2 [v] + V_3 [v] + V_4 [v].% + V_5 [v].% + V_6.% + V_7.
\end{equation}
Moreover, $v$ satisfies the following initial data and Dirichlet boundary conditions:
\begin{equation*}
    \restriction{v}{t = 0} = \partial_x \varphi, \qquad v(t,-a) = v(t,a) = 0.
\end{equation*}

\begin{remark} \label{rem:not_regu_comp}
    Essentially, our computations make use of the fact that we can decompose the logarithmic nonlinearity into two separate cases:
    \begin{equation*}
        u \log{\abs{u}^2} = u \log{\abs{x}^2} + u \log{\abs{\mathcal{V} [\partial_x u]}^2}.
    \end{equation*}
    The first part is linear in $u$, and thus it is treated directly in the operator $A$ as a potential (up to a localization, making this potential compactly supported), whereas the second term is treated as a new nonlinear term, depending on $\partial_x u$ but enjoying favorable properties, especially after spatial differentiation. \\
   These computations explain why a regularized version of the equation is not suitable for such analysis: with a saturated logarithm ($\log{(\abs{u}^2 + \varepsilon)}$ or so), this separation of terms is no longer feasible.
    More importantly, several new terms involving $\varepsilon$ appear, which cannot be bounded uniformly as $\varepsilon \to 0$.
\end{remark}

% \begin{multline*}
%     v (t) = e^{- i t A} \partial_x \varphi + 2 \lambda \int_0^t e^{- i (t-s) A} \Bigl( \mathcal{V} [v(s)] \frac{\chi'}{\mathfrak{X}} \Bigr) \dd s \\
%         \begin{aligned}
%             &+ \lambda \int_0^t e^{- i (t-s) A} \Bigl( v \log{\abs{\mathcal{V} [v(s)]}^2} \Bigr) \dd s \\
%             &+ 2 \lambda \int_0^t e^{- i (t-s) A} \Biggl( \mathcal{V} [v(s)] \frac{\Re \Bigl( v (s) \, \overline{\mathcal{V} [v(s)]} \Bigr)}{\abs{\mathcal{V} [v(s)]}^2} \Biggr) \dd s \\
%             &- 2 \lambda \int_0^t e^{- i (t-s) A} \Biggl( \mathcal{V} [v(s)] \frac{\Re \Bigl( \mathcal{V} [v (s)] \, \overline{\mathcal{V} [v(s)]} \Bigr)}{\abs{\mathcal{V} [v(s)]}^2} \Biggr) \dd s \\
%             &- 2 \lambda \int_0^t e^{- i (t-s) A} \Bigl( v \log \mathfrak{X} \Bigr) \dd s \\
%             &- 2 \lambda \int_0^t e^{- i (t-s) A} \Bigl( u \frac{\mathfrak{X}'}{\mathfrak{X}} \Bigr) \dd s \\
%             &- \kappa \int_0^t e^{- i (t-s) A} u (s) \dd s.
%         \end{aligned}
% \end{multline*}
%

\subsection{Properties of the $V_i$}

% We start with the easiest term to estimate: $V_6$.

% \begin{lemma}
%     For every $T > 0$, there exists $C_T > 0$ such that, for all $t \in [0, T]$,
%     \begin{equation*}
%         \norm{V_6}_{H^2} \leq C_T.
%     \end{equation*}
% \end{lemma}

% \begin{proof}
%     This is due to the facts that:
%     \begin{itemize}
%         \item $u$ is bounded in $H^2$, uniformly on every bounded interval of time,
%         \item $\mathfrak{X}$ is $\mathcal{C}^2$ and uniformly bounded away from $0$. \qedhere
%     \end{itemize}
% \end{proof}

%$V_5$ is also very easy to estimate.

We first focus our attention on the $V_i$ terms. Using once again the fact that $\mathfrak{X}$ is $\mathcal{C}^2$ and bounded away from $0$, we can show the following properties as a straightforward consequence from the Lemmas \ref{lem:prod_Sobolev}, \ref{lem:special_Sobolev}, \ref{lem:diff_special_Sobolev} and Corollary \ref{cor:xV_sobolev}:
\begin{lemma} \label{lem:est_Vj_Sobolev}
    Let $R, \alpha > 0$. There exists $C_{R, \alpha, 1}, C_{R, \alpha, 2} > 0$ such that the following holds.
    \begin{itemize}
        \item For $k \in \{ 1, 2\}$ and for any $f_\ell \in H^k (\Omega)$ ($\ell \in \{ 1, 2 \}$) such that $\norm{f_\ell}_{H^1} \leq R$ and $\abs{\mathcal{V} [f_\ell]} \geq \alpha$, there holds for all $1 \leq j \leq 4$
        \begin{gather*}
            \norm{V_j [f_\ell]}_{H^k} \leq C_{R, \alpha, 1} \norm{f_\ell}_{H^k}, \\
            \norm{V_j [f_1] - V_j [f_2]}_{H^1} \leq C_{R, \alpha, 1} \norm{f_1 - f_2}_{H^1}.
        \end{gather*}
        \item For any $f_\ell \in H^2 (\Omega)$ ($\ell \in \{ 1, 2 \}$) such that $\norm{f_\ell}_{H^2} \leq R$ and $\abs{\mathcal{V} [f_\ell]} \geq \alpha$, there holds for all $1 \leq j \leq 4$
        \begin{equation*}
            \norm{V_j [f_1] - V_j [f_2]}_{H^2} \leq C_{R, \alpha, 2} \norm{f_1 - f_2}_{H^2}.
        \end{equation*}
    \end{itemize}
\end{lemma}

% PLUS BESOIN DE LA PREUVE CI-DESSOUS, MAIS JE LA LAISSE POUR LE MOMENT.

% ON POURRA METTRE TOUT OU PARTIE DES LEMMES CITES AU-DESSUS EN APPENDICE POUR "CACHER" LES CALCULS TECHNIQUES ET SE CONCENTRER SUR LE COEUR DE LA PREUVE.

\begin{proof}
    All these estimates can be achieved thanks to Lemma \ref{lem:computations_Sobolev}, Lemma \ref{lem:V_sobolev} and Corollary \ref{cor:xV_sobolev}.
    For clarity and brevity, we will only estimate $V_3 [f]$ (which is the hardest case).
    From the assumption $\norm{f}_{H^1} \leq R$ and from Lemma \ref{lem:computations_Sobolev}, Lemma \ref{lem:V_sobolev} and Corollary \ref{cor:xV_sobolev}, we can find $R_0 > 0$ (depending only on $R$) such that
    \begin{equation*}
        \norm{\mathcal{V} [f]}_{H^1} + \norm{\mathcal{V} [f]}_{L^\infty} \leq R_0
    \end{equation*}
    and such that, by using using Lemma \ref{lem:computations_Sobolev},
    \begin{gather*}
        \norm{\mathcal{V} [f] \Re \Bigl( f \, \overline{\mathcal{V} [f]} \Bigr)}_{H^1} \leq R_0, \\
        \norm{\mathcal{V} [f] \Re \Bigl( f \, \overline{\mathcal{V} [f]} \Bigr)}_{H^2} \leq R_0 \norm{f}_{H^2}.
    \end{gather*}
    Therefore, we can estimate as follows, by using the fact that $\abs{\mathcal{V} [f]} \geq \alpha$:
    \begin{align*}
        \norm{V_3 [f]}_{H^1} &= 2 \abs{\lambda} \norm{\frac{\mathcal{V} [f] \Re \Bigl( f \, \overline{\mathcal{V} [f]} \Bigr)}{\abs{\mathcal{V} [f]}^2}}_{H^1} \\
            &\leq 2 \abs{\lambda} \frac{\norm{\mathcal{V} [f] \Re \Bigl( f \, \overline{\mathcal{V} [f]} \Bigr)}_{L^2}}{\alpha^2} + 2 \abs{\lambda} \frac{C_1}{\alpha^3} \norm{\mathcal{V} [f] \Re \Bigl( f \, \overline{\mathcal{V} [f]} \Bigr)}_{H^1} \norm{\mathcal{V} [f]}_{H^1} \\
            &\leq 2 \abs{\lambda} \frac{R_0}{\alpha^2} + 2 \abs{\lambda} \frac{C_1}{\alpha^3} R_0^2,
    \end{align*}
    and also
    \begin{align*}
        \norm{\partial_{xx} V_3 [f]}_{L^2} &= 2 \abs{\lambda} \norm{\partial_{xx} \biggl( \frac{\mathcal{V} [f] \Re \Bigl( f \, \overline{\mathcal{V} [f]} \Bigr)}{\abs{\mathcal{V} [f]}^2} \biggr)}_{L^2} \\
            &\leq
            \begin{multlined}[t]
                2 \abs{\lambda} \frac{C_2}{\alpha^4} \norm{\mathcal{V} [f]}_{H^1} \Bigl( \norm{\mathcal{V} [f] \Re \Bigl( f \, \overline{\mathcal{V} [f]} \Bigr)}_{H^1} \norm{\mathcal{V} [f]}_{H^2} \\ + \norm{\mathcal{V} [f] \Re \Bigl( f \, \overline{\mathcal{V} [f]} \Bigr)}_{H^2} \norm{\mathcal{V} [f]}_{H^1} \Bigr)
            \end{multlined}
             \\
            &\leq 2 \abs{\lambda} \frac{C_2}{\alpha^4} R_0 \left( R_0 \left(C_{\Omega, a} + \frac{5}{2} \right) \norm{f}_{H^2} + R_0^2 \norm{f}_{H^2} \right)
            \\
            &\leq 2 \abs{\lambda} \frac{C_2}{\alpha^4} R_0^2 \left( \left(C_{\Omega, a} + \frac{5}{2} \right) + R_0 \right) \norm{f}_{H^2}
    \end{align*}
    which leads to the conclusion.
\end{proof}

\begin{corollary} \label{cor:est_Schrogroup_Vj}
    Let $R, \alpha > 0$. There exists $K_{R, \alpha, 1}, K_{R, \alpha, 2} > 0$ such that the following holds.
    \begin{itemize}
        \item For $k \in \{ 1, 2\}$ and for any $f_\ell \in H^k (\Omega)$ ($\ell \in \{ 1, 2 \}$) such that $\norm{f_\ell}_{H^1} \leq R$ and $\abs{\mathcal{V} [f_j]} \geq \alpha$, there holds for all $1 \leq j \leq 5$ and all $\tau \in \mathbb{R}$
        \begin{gather*}
            \norm{e^{i \tau A} V_j [f_\ell]}_{H^k} \leq K_{R, \alpha, 1} \norm{f_\ell}_{H^k}, \\
            \norm{e^{i \tau A} (V_j [f_1] - V_j [f_2])}_{H^1} \leq K_{R, \alpha, 1} \norm{f_1 - f_2}_{H^1}.
        \end{gather*}
        \item For any $f_\ell \in H^2 (\Omega)$ ($\ell \in \{ 1, 2 \}$) such that $\norm{f_\ell}_{H^2} \leq R$ and $\abs{\mathcal{V} [f_j]} \geq \alpha$, there holds for all $1 \leq j \leq 5$ and all $\tau \in \mathbb{R}$
        \begin{equation*}
            \norm{e^{i \tau A} (V_j [f_1] - V_j [f_2])}_{H^2} \leq K_{R, \alpha, 2} \norm{f_1 - f_2}_{H^2}.
        \end{equation*}
    \end{itemize}
\end{corollary}

\begin{proof}
    This result is achieved by applying consecutively Proposition \ref{prop:evol_group_schrod_A} and Lemma \ref{lem:est_Vj_Sobolev}.
\end{proof}

Moreover, we show that, if $f \in H^1_0 (\Omega)$, then so do the $V_j [f]$.

\begin{lemma} \label{lem:Vj_H10}
    Let $f \in H^1_0 (\Omega)$ such that $\abs{\mathcal{V} [f]} \geq \alpha$ for some $\alpha > 0$. Then, for all $1 \leq j \leq 4$, we have $V_j [f] \in H^1_0 (\Omega)$.
\end{lemma}

\begin{proof}
    Since we already proved that all the $V_j$ are in $H^1 (\Omega) \hookrightarrow C^0 (\overline{\Omega})$ (in dimension $1$), we only have to check that all these terms vanish at the boundaries, that is, at $y = a$ and $y = -a$. Since $\abs{\mathcal{V} [f]} \geq \alpha > 0$ and $1 \geq \mathfrak{X} \geq C_\chi > 0$, the following estimates hold for all $y \in [-a, a]$:
    \begin{gather*}
        \abs{V_2 [f] (y)} \leq \abs{\lambda} \abs{f (y)} \Bigl( \abs{\log{\alpha^2}} + \abs{\log{\norm{f}_{L^\infty}^2}} \Bigr), \\
        \abs{V_3 [f] (y)} \leq 2 \abs{\lambda} \abs{f (y)}, \\
        \abs{V_4 [f] (y)} \leq 2 \abs{\lambda} \abs{\log{C_\chi}} \abs{f (y)}.
    \end{gather*}
    Since $\lim_{\pm a} f = 0$ by the fact that $f \in H^1_0 (\Omega)$, the conclusion follows for these cases. As for $V_1 [f]$, the fact that both $\chi'$ and $1 - x \frac{\mathfrak{X}'}{\mathfrak{X}}$ vanish outside~$[-2, 2]$ leads to $V_1 [f] \equiv 0$ outside~$[-2, 2]$, and in particular on the boundary.
\end{proof}

\subsection{Duhamel formulation}

We now state the rigorous Duhamel's formulation for equation~\eqref{eq:v}:

\begin{lemma} \label{lem:link_duhamel}
    Let $\varphi \in \mathcal{D} \cap H^2 (\Omega)$ such that $\partial_x \varphi \in H^1_0(\Omega)$ and denote $u \in \mathcal{C}(\R, H^2 (\Omega))$ the unique solution to \eqref{logNLS} with initial data $u(0)=\varphi$ and Neumann boundary conditions. Let~$v = \partial_x u$ and let $\alpha, T > 0$ be defined by Lemma \ref{lemma_minoration_dx_u} such that $\abs{\mathcal{V} [v (t)]} \geq \alpha$ for all $t \in (-T, T)$. Then there holds for all $t \in (-T,T)$
    \begin{equation} \label{eq:duhamel_dxu}
        v (t) = e^{- i t A} \partial_x \varphi -i \int_0^t e^{- i (t-s) A} \sum_{j=1}^4 V_j [v(s)] \dd s.
    \end{equation}
\end{lemma}
\begin{proof}
    Let's first note that all the terms involved in the right hand side of equation \eqref{eq:duhamel_dxu} are well defined thanks to Proposition \ref{prop:evol_group_schrod_A} and Lemma \ref{lem:est_Vj_Sobolev}, since $v \in \mathcal{C} ((-T,T); H^1(\Omega))$.
    % Differentiating this expression in time, it is then clear that it is solution of equation \eqref{eq:v}.\\
    % Conversely, let $v \in \mathcal{C} ((-T,T); H^1(\Omega))$ be defined as the solution of equation \eqref{eq:v}, and denote
    From the computations at the beginning of this section, we know that $v$ satisfies \eqref{eq:v}.
    Let $\tilde{v}=e^{itA}v \in \mathcal{C} ((-T,T); H^1(\Omega))$ in view of Proposition \ref{prop:evol_group_schrod_A}. Then 
    \[  i \partial_t \tilde{v}(t) =  e^{-itA}\sum_{j=1}^4 V_j [v(t)] \]
    for all time $t \in \left(-T,T \right)$, with the right-hand side being in $L^\infty ((-T, T); H^1_0 (\Omega))$. By integrating over time, one gets
    \[  \tilde{v}(t)=e^{ i t A}v (t) = e^{ i t A} \partial_x \varphi -i \int_0^t e^{is A} \sum_{j=1}^4 V_j [v(s)] \dd s,   \]
    which gives the result applying $e^{-itA}$.
\end{proof}

One can state an immediate corollary:
\begin{corollary} \label{cor:link_duhamel_dxu}
    Let $w = v - \partial_x \varphi$. Then there holds
    \begin{equation} \label{eq:duhamel_dxu_diff}
        w (t) = e^{- i t A} \partial_x \varphi - \partial_x \varphi -i \int_0^t e^{- i (t-s) A} \sum_{j=1}^4 V_j [w(s) + \partial_x \varphi] \dd s.
    \end{equation}
\end{corollary}

\subsection{Properties of the nonlinear part}

We define
\begin{gather} \label{eq:def_N_mcalN}
    N [f] (\tau) \coloneqq e^{- i \tau A} \sum_{j=1}^4 V_j [f], \\
    \mathcal{N} [v] (t) \coloneqq \int_0^t e^{- i (t-s) A} \sum_{j=1}^4 V_j [v (s)] \dd s = \int_0^t N [v (s)] (t-s) \dd s = e^{- i t A} \int_0^t N [v (s)] (-s) \dd s.
\end{gather}

We then show some smallness and Lipschitz properties for $\mathcal{N}$:

\begin{lemma} \label{lem:est_mcalN}
    Let $R, \alpha > 0$. There exists $C_{R, \alpha} > 0$ such that the following holds:
    \begin{itemize}
        \item Let $v_\ell \in \mathcal{C} ((- T, T), H^1_0 (\Omega))$ ($\ell = 1, 2$) such that $\abs{\mathcal{V} [v_\ell (t)]} \geq \alpha$ and $\norm{v_\ell (t)}_{H^1} \leq R$ for all $t \in (- T, T)$.
    Then $(\tau, t) \mapsto N [v_\ell (t)] (\tau) \in L^\infty (\mathbb{R} \times (-T, T); H^1_0 (\Omega))$ and $\mathcal{N} [v_\ell] \in \mathcal{C} ((- T, T), H^1_0 (\Omega))$, and they satisfy
    \begin{gather}
        \norm{N [v_\ell (s)] (\tau)}_{H^1} \leq C_{R, \alpha}, \qquad \forall \tau \in \mathbb{R}, \forall s \in (-T, T), \label{eq:est_N1} \\
        \norm{\mathcal{N} [v_\ell] (t)}_{H^1} \leq C_{R, \alpha} T, \qquad \forall t \in (-T,T), \label{eq:est_mcalN_1} \\
        \norm{\mathcal{N} [v_1] (t) - \mathcal{N} [v_2] (t)}_{H^1} \leq C_{R, \alpha} T \norm{v_1 - v_2}_{L^\infty_T H^1}, \qquad \forall t \in (-T,T). \label{eq:est_mcalN_2}
    \end{gather}
    \item Furthermore, if $v_\ell \in \mathcal{C}((- T, T), H^1_0 (\Omega) \cap H^2 (\Omega))$, there also hold $(\tau, t) \mapsto N [v_\ell (t)] (\tau) \in L^\infty (\mathbb{R} \times (-T, T); H^1_0 (\Omega) \cap H^2 (\Omega))$ and $\mathcal{N} [v_\ell] \in \mathcal{C}((-T,T), H^1_0 (\Omega) \cap H^2 (\Omega))$, as well as
    \begin{gather}
        \norm{N [v_\ell (s)] (\tau)}_{H^2} \leq C_{R, \alpha} \norm{v_\ell (s)}_{H^2}, \qquad \forall \tau \in \mathbb{R}, \forall s \in (-T, T), \notag \\
        \norm{\mathcal{N} [v_\ell] (t)}_{H^2} \leq C_{R, \alpha} \int_0^t \norm{v_\ell (s)}_{H^2} \dd s \leq C_{R, \alpha} T \norm{v_\ell}_{L^\infty_T H^2}, \qquad \forall t \in (-T,T). \label{eq:est_mcalN_3}
    \end{gather}
    \item Last, if we also have $\norm{v_\ell (t)}_{H^2} \leq R$ for all $t \in (- T, T)$, then there holds
    \begin{gather*}
        \norm{N [v_1 (s)] (\tau) - N [v_2 (s)] (\tau)}_{H^2} \leq C_{R, \alpha} \norm{v_1 (s) - v_2 (s)}_{H^2}, \qquad \forall \tau \in \mathbb{R}, \forall s \in (-T, T), \\
        \norm{\mathcal{N} [v_1] (t) - \mathcal{N} [v_2] (t)}_{H^2} \leq C_{R, \alpha} T \norm{v_1 - v_2}_{L^\infty_T H^2}, \qquad \forall t \in (-T,T).
    \end{gather*}
    \end{itemize}
\end{lemma}

\begin{proof}
    For the $H^1_0$ case, with the first assumptions, we can apply Corollary \ref{cor:est_Schrogroup_Vj} with general $\tau \in \mathbb{R}$, so that the estimate \eqref{eq:est_N1} can be found after summation. Then, we can apply the latter estimate twice:
    \begin{itemize}
        \item first with $\tau = t-s$, so that \eqref{eq:est_mcalN_1} and \eqref{eq:est_mcalN_2} are found after integration in time,
        \item secondly, with $\tau = - s$, leading to the continuity of $t \mapsto \int_0^t N [v_\ell (s)] (-s) \dd s$ in $(- T, T)$ with values in $H^1_0$, and finally the continuity of $\mathcal{N} [v_\ell]$ in $(- T, T)$ with values in $H^1_0$ thanks to Theorem \ref{prop:evol_group_schrod_A} and the last relation in \eqref{eq:def_N_mcalN}.
        % \begin{equation*}
        %     \mathcal{N} [v_\ell] (t) = e^{- i t A} \int_0^t N [v_\ell (s)] (-s) \dd s.
        % \end{equation*}
    \end{itemize}
    The other estimates of this result are obtained with similar arguments.
\end{proof}

\subsection{Cauchy theory}

We aim to establish a well-posed Cauchy theory for equation~\eqref{eq:duhamel_dxu} in both $H^1_0(\Omega)$ and $H^1_0(\Omega) \cap H^2(\Omega)$. However, we know that certain terms $V_j [v(s)]$ (notably $V_2 [v(s)]$ and $V_3 [v(s)]$) are only well-defined when $\abs{\mathcal{V} [v] (s)}$ does not vanish. As a result, it is not possible to develop a Cauchy theory on the entire space. Given that the initial data already satisfies the necessary assumptions and since we expect some continuity, a natural way to address this issue is to work within a (small) ball in the aforementioned spaces, centered on the initial data. Within this framework, we are restricted to a short time interval, which is not very surprising as we do not expect a global Cauchy theory with such techniques.\\

We define, for $T > 0$,
\begin{gather*}
    Y_{1, \varepsilon, T} \coloneqq \mathcal{C}( \left[-T,T\right], B_{H^1_0 (\Omega)} (0, \varepsilon)), \\
    Y_{2, \varepsilon, T} \coloneqq \mathcal{C}(\left[-T,T\right], B_{H^1_0 (\Omega) \cap H^2 (\Omega)} (0, \varepsilon)),
\end{gather*}
where
\begin{gather*}
    B_{H^1_0 (\Omega)} (0, \varepsilon) = \{ f \in H^1_0 (\Omega) \mid \norm{f}_{H^1} \leq \varepsilon \}, \\
    B_{H^1_0 (\Omega) \cap H^2 (\Omega)} (0, \varepsilon) = \{ f \in H^1_0 (\Omega) \cap H^2 (\Omega) \mid \norm{f}_{H^2} \leq \varepsilon \}.
\end{gather*}

Our goal is to establish a well-posed Cauchy theory in both $Y_{1, \varepsilon_1, T_1}$ and $Y_{2, \varepsilon_2, T_2}$, associated to some~$\varepsilon_{\ell}, T_{\ell} > 0$ for
\begin{equation} \label{eq:abstract_rel_for_Cauchy}
    w (t) = e^{- i t A} v_0 - v_0 + \mathcal{N} [v_0 + w] (t),
\end{equation}
with initial data $v_0$ (satisfying similar assumptions as $\partial_x \varphi$) at $t=0$, so that $v \coloneqq v_0 + w$ satisfies
\begin{equation} \label{eq:abstract_rel_for_Cauchy_2}
    v (t) = e^{- i t A} v_0 + \mathcal{N} [v] (t).
\end{equation}
\begin{remark} \label{rem:time_inv_abstract}
    We note that, since the equation is invariant under time translations, all forthcoming results remain valid for initial data prescribed at any time $t = t_0$ with the usual changes. This applies in particular to the Cauchy theory, which will be addressed in Proposition \ref{prop:Cauchy_abstract}.
\end{remark}

We first show that, as soon as $\varepsilon > 0$ is small enough, this equation is well defined.

\begin{lemma} \label{lem:est_Vcal_bfb}
    Let $v_0 \in H^1_0 (\Omega)$ satisfying $\abs{\mathcal{V} [v_0]} \geq \alpha$ for some $\alpha > 0$. Then there exists $\varepsilon_0 > 0$ such that, for any $T > 0$, for any $w \in Y_{k, \varepsilon_0, T}$ (with $k=1$ or $k=2$) and for all $t \in\left[-T,T\right]$,
    \begin{gather*}
        \norm{v_0 + w(t)}_{H^k} \leq R_{k, \varepsilon_0}, \\
        \norm{\mathcal{V} [v_0 + w (t)]}_{H^k} \leq C_{\Omega} R_{k, \varepsilon_0}, 
    \end{gather*}
    and
    \begin{equation} \label{eq:mathcal_V}
         \abs{\mathcal{V} [v_0 + w (t)]} \geq \frac{\alpha}{2},
    \end{equation}
    where $R_{k, \alpha, \varepsilon_0} \coloneqq \norm{v_0}_{H^k} + \varepsilon_0$.
\end{lemma}

\begin{proof}
    The first and second estimates are direct as $w \in Y_{k,\eps_0,T}$ and from Corollary \ref{cor:xV_sobolev}. From Lemma \ref{lem:mcalV_Linfty}, we also get that
    \begin{equation*}
        \abs{\mathcal{V} [w (t)]} \leq \norm{w}_{L^\infty} \leq  C_{\Omega} \norm{w}_{H^1} \leq  C_{\Omega} \varepsilon,
    \end{equation*}
    from Sobolev embeddings. Assuming that $C_{\Omega} \varepsilon < \frac{\alpha}{2}$, we get the conclusion by assumptions and using triangular inequality.
\end{proof}

We now have all the material to prove the Cauchy theory of \eqref{eq:abstract_rel_for_Cauchy}.

\begin{proposition} \label{prop:Cauchy_abstract}
    Let $v_0 \in H^1_0 (\Omega)$ (resp. $v_0 \in H^1_0 (\Omega) \cap H^2 (\Omega)$) satisfying $\abs{\mathcal{V} [v_0]} \geq \alpha$ for some $\alpha > 0$. Then there exists $\varepsilon_0 > 0$ depending on $\alpha$ such that, for all $\varepsilon \in (0, \varepsilon_0)$, the following holds: there exists $T_0 > 0$ such that, for all $T \in (0, T_0)$, there exists a unique $w \in Y_{1, \varepsilon, T}$ (resp. $w \in Y_{2, \varepsilon, T}$) such that, for all $t \in \left[-T,T\right]$, equation \eqref{eq:abstract_rel_for_Cauchy} holds. Moreover, equation \eqref{eq:mathcal_V} holds for all $t \in \left[-T,T\right]$.
\end{proposition}

\begin{proof}
    We assume that $0 < \varepsilon < \varepsilon_0$ where $\varepsilon_0 > 0$ is given by Lemma \ref{lem:est_Vcal_bfb}. In particular, we can apply Lemma \ref{lem:est_mcalN}, so that there exists $K_{\varepsilon_0, \alpha} > 0$ such that, for $k \in \{ 1, 2 \}$, for $w_\ell \in Y_{k, \varepsilon, T}$ ($\ell \in \{ 1, 2 \}$), we have
    \begin{gather*}
        \norm{\mathcal{N} [v_0 + w_\ell] (t)}_{H^1} \leq K_{\varepsilon_0, \alpha} T, \\
        \norm{\mathcal{N} [v_0 + w_1] (t) - \mathcal{N} [v_0 + w_2] (t)}_{H^1} \leq K_{\varepsilon_0, \alpha} T \norm{w_1 - w_2}_{L^\infty_T H^1}.
    \end{gather*}
    
    On the other hand, Theorem \ref{prop:evol_group_schrod_A} shows that $t \mapsto e^{-i t A} v_0 - v_0$ is continuous with values in $H^1_0$ if $v_0 \in H^1_0 (\Omega)$ (resp. $H^1_0(\Omega) \cap H^2(\Omega)$ if $v_0 \in H^1_0 (\Omega) \cap H^2 (\Omega)$), so that there exists $T_0 > 0$ such that, for all $t \in [- T_0, T_0]$,
    \begin{equation*}
        \norm{e^{-i t A} v_0 - v_0}_X < \frac{\varepsilon}{2},
    \end{equation*}
    where $X = H^1_0$ if $v_0 \in H^1_0 (\Omega)$ (resp. $X = H^2$ if $v_0 \in H^1_0 (\Omega) \cap H^2 (\Omega)$).
    
    Therefore, by taking $\varepsilon < \varepsilon_0$ and then $T < T_0$ such that $K_{\varepsilon_0, \alpha} T < \frac{\varepsilon}{2} < 1$, for $k = 1, 2$, the application
    \begin{align*}
        Y_{k, \varepsilon, T} &\rightarrow Y_{k, \varepsilon, T} \\
        w &\mapsto \Bigl( t \mapsto e^{-i t A} v_0 - v_0 + \mathcal{N} [v_0 + w] (t) \Bigr)
    \end{align*}
    is a contraction in $Y_{k, \varepsilon, T}$ if $v_0 \in H^1_0$ (for $k=1$) or $v_0 \in H^1_0 \cap H^2$ (for $k=2$), leading to the conclusion by Picard fixed point theorem in Banach spaces.
\end{proof}

Now that the Cauchy theory for equation \eqref{eq:abstract_rel_for_Cauchy} is established, we can turn our attention to equation \eqref{eq:abstract_rel_for_Cauchy_2}.

\begin{lemma} \label{lem:H2_bound_abstract}
    Let $v \in \mathcal{C}([- T, T]; H^1_0 (\Omega) \cap H^2 (\Omega))$ a solution of \eqref{eq:abstract_rel_for_Cauchy_2} for some $v_0 \in H^1_0 (\Omega) \cap H^2 (\Omega)$ satisfying $\abs{\mathcal{V} [v (t)]} \geq \alpha$ for all $t \in [- T, T]$ for some $\alpha > 0$.
    Let $R = \sup_{t \in [- T, T]} \norm{v (t)}_{H^1}$.
    Then there holds for all $t \in [- T, T]$
    \begin{equation*}
        \norm{v (t)}_{H^2} \leq \Bigl( C_2 + \frac{\sqrt{C_1}}{\sqrt{c_1}} (1 + c_2) \Bigr) \norm{v_0}_{H^2} \, e^{C_{R, \alpha} t}.
    \end{equation*}
\end{lemma}

\begin{proof}
    By the assumptions, we can apply Lemma \ref{lem:est_mcalN}, and in particular \eqref{eq:est_mcalN_3}, so that we get with \eqref{eq:abstract_rel_for_Cauchy_2} and with Theorem \ref{prop:evol_group_schrod_A}:
    \begin{align*}
        \norm{v (t)}_{H^2} &\leq \norm{e^{- i t A} v_0}_{H^2} + \norm{\mathcal{N} [v] (t)}_{H^2} \\
            &\leq \Bigl( C_2 + \frac{\sqrt{C_1}}{\sqrt{c_1}} (1 + c_2) \Bigr) \norm{v_0}_{H^2} + C_{R, \alpha} \int_0^t \norm{v (s)}_{H^2} \dd s.
    \end{align*}
    The conclusion is achieved by Gronwall's Lemma.
\end{proof}

We are now in a position to establish the propagation of regularity for equation \eqref{eq:abstract_rel_for_Cauchy_2}:

\begin{corollary} \label{cor:prop_H2_if_H1}
    Let $T_1 > T_2 > 0$ and $v \in \mathcal{C}([- T_1, T_1]; H^1_0 (\Omega)) \cap \mathcal{C}((- T_2, T_2); H^2 (\Omega))$ a solution of \eqref{eq:abstract_rel_for_Cauchy_2} and assume that $\abs{\mathcal{V} [v (t)]} \geq \alpha$ for all $t \in [- T_1, T_1]$ for some $\alpha > 0$. Then we have $v \in \mathcal{C}([- T_1, T_1]; H^2 (\Omega))$.
\end{corollary}

\begin{proof}
    We define
    \begin{equation*}
        T_0 \coloneqq \sup \{ T \in (0, T_1] \mid v \in \mathcal{C}((- T, T); H^2 (\Omega)) \} \geq T_2 > 0.
    \end{equation*}
    Then $v \in \mathcal{C}((- T_0, T_0); H^2 (\Omega))$. Let $R \coloneqq \sup_{t \in [- T_1, T_1]} \norm{v (t)}_{H^1}$ and apply Lemma \ref{lem:H2_bound_abstract}, so that for all $t \in (- T_0, T_0)$
    \begin{equation*}
        \norm{v (t)}_{H^2} \leq \Bigl( C_2 + \frac{\sqrt{C_1}}{\sqrt{c_1}} (1 + c_2) \Bigr) \norm{v_0}_{H^2} \, e^{C_{R, \alpha} t}.
    \end{equation*}
    From this estimate along with Lemma \ref{lem:est_mcalN} (which implies the uniform boundedness of $N [v (s)] (\tau)$ in $H^2$ for all $\tau \in \mathbb{R}$ and $s \in (- T_0, T_0)$), we get that 
    \[ \mathcal{N} [v] (t) = \int_0^t N [v (s)] (t-s) \dd s \in \mathcal{C}([- T_0, T_0], H^2 (\Omega)), \] 
    which in turn implies that $v \in \mathcal{C} ([- T_0, T_0], H^2 (\Omega))$ with \eqref{eq:abstract_rel_for_Cauchy}.

    To conclude, we need to prove that $T_0 = T_1$. By contradiction, if $T_0 < T_1$, then Proposition~\ref{prop:Cauchy_abstract} applied with the initial data $v (T_0) \in H^1_0 (\Omega) \cap H^2 (\Omega)$ shows that there exists $\delta t > 0$ such that $v \in \mathcal{C} ([T_0 - \delta t, T_0 + \delta t], H^1_0 (\Omega) \cap H^2 (\Omega))$. A similar argument can be performed for $t = -T_0$, leading to an obvious contradiction regarding the definition of $T_0$.
\end{proof}

We now turn to the question of existence and uniqueness for equation \eqref{eq:abstract_rel_for_Cauchy_2}:

\begin{corollary} \label{cor:Cauchy_less_abstract}
    Let $v_0 \in H^1_0 (\Omega)$ such that $\abs{\mathcal{V} [v_0]} \geq \alpha$ for some $\alpha > 0$.
    There exists~$T_0 > 0$ such that, for all $T \in (0, T_0)$, there exists a unique solution $v \in \mathcal{C}([- T, T]; H^1_0 (\Omega))$ to \eqref{eq:duhamel_dxu}. Furthermore, if $v_0 \in H^2 (\Omega)$, then we also have $v \in \mathcal{C}([- T, T]; H^2 (\Omega))$.
\end{corollary}

\begin{proof}
    We get the existence on $H^1_0(\Omega)$ (resp. $H^1_0(\Omega) \cap H^2(\Omega)$ by applying Proposition \ref{prop:Cauchy_abstract} taking any $\varepsilon \in (0, \varepsilon_0)$, where $\varepsilon_0 > 0$ is given by Proposition \ref{prop:Cauchy_abstract} (for instance, take $\varepsilon = \frac{\varepsilon_0}{2}$), and by defining $v = v_0 + w$.
    
    We turn our attention to the uniqueness. Let $T \in (0, T_0)$ and $v_1 \in \mathcal{C}([- T, T]; H^1_0 (\Omega))$ be another solution of equation \eqref{eq:abstract_rel_for_Cauchy_2}. Define $w_1 = v_1 - v_0$. From the assumption, we have
    \begin{equation*}
        T_* \coloneqq \sup \{ T_1 > 0 \mid \forall t \in (- T_1, T_1), \norm{w_1 (t)}_{H^1} < \frac{\varepsilon_0}{2} \} > 0.
    \end{equation*}
    We want to prove that $T_* = T$ and $w_1 = w$. Let $T_1 < T_*$. From the assumption, $w_1$ is a solution of \eqref{eq:duhamel_dxu_diff} and belongs to $Y_{1, \frac{\varepsilon_0}{2}, T_1}$. Since $T_1 < T_0$, we can apply the uniqueness part of Proposition \ref{prop:Cauchy_abstract}, which gives $\restriction{w_1}{(-T_1, T_1)} = \restriction{w}{(-T_1, T_1)}$.
    
    Assume now that $T_* < T$.
    Since the previous equality is true for any $T_1 < T_* < T$, and by continuity of both $w$ and $w_1$, we get $w_1 (T_*) = w (T_*)$. Thus, $\norm{w_1 (T_*)}_{H^1} = \norm{w (T_*)}_{H^1} < \frac{\varepsilon_0}{2}$, and the continuity of $w_1$ gives a classical contradiction with the definition of $T_*$. The conclusion follows.

    As for the $H^2$ regularity, if $v_0 \in H^2$, then there also exists by Proposition \ref{prop:Cauchy_abstract} a unique solution $w_2 \in Y_{2, \varepsilon_2, T_2}$ to \eqref{eq:duhamel_dxu_diff} for some $\varepsilon_2 > 0$ small enough and $T_2 > 0$ small enough, and we can assume that $\varepsilon_2 < \frac{\varepsilon_0}{2}$ and $T_2 < T$. Then we also have $w_2 \in Y_{1, \varepsilon_0, T_2}$, and by uniqueness in Proposition \ref{prop:Cauchy_abstract} we get $w_2 = \restriction{w}{[-T_2, T_2]}$, which means that $v \in \mathcal{C}([- T_2, T_2]; H^2 (\Omega))$. Moreover, as $|\mathcal{V}\left[v(t)\right] | \geq \alpha$ for all $t \in \left[-T,T \right]$ by Proposition \ref{prop:Cauchy_abstract}, and as $v \in \mathcal{C}([- T, T]; H^1_0 (\Omega))$, we can apply Corollary \ref{cor:prop_H2_if_H1}, so $v \in \mathcal{C}((- T, T); H^2 (\Omega))$ which ends the proof.
    
    % \textcolor{blue}{Therefore, we define
    % \begin{equation*}
    %     T_3 \coloneqq \sup \{ \overline{T} \leq T \mid v \in \mathcal{C}([- \overline{T}, \overline{T}]; H^2 (\Omega)) \} \geq T_2 > 0,
    % \end{equation*}
    % %
    % so that we directly see that $v \in \mathcal{C}((- T_3, T_3); H^2 (\Omega))$. Since $\abs{\mathcal{V} [v (t)]} \geq \frac{\alpha}{2}$ by construction, applying Lemma \ref{lem:H2_bound_abstract} leads to the estimate
    % \begin{equation*}
    %     \norm{v (t)}_{H^2} \leq \Bigl( C_2 + \frac{\sqrt{C_1}}{\sqrt{c_1}} (1 + c_2) \Bigr) \norm{v_0}_{H^2} \, e^{C_{R, \alpha} t}, \qquad \forall t \in (- T_3, T_3).
    % \end{equation*}
    % %
    % By contradiction, if $T_3 < T$, since $v \in \mathcal{C}([- T, T]; H^1_0 (\Omega))$, this bound proves that $v (\pm T) \in H^2 (\Omega)$. Similar arguments with initial data $v (\pm T)$ show that we can extend the $H^2$ regularity beyond $T$ and below $- T$, which contradicts the definition of $T_3$. Therefore $v \in \mathcal{C}((- T, T); H^2 (\Omega))$.}
\end{proof}

We can now turn to the main proposition of this section, which induces in particular Theorem \ref{theo:propagation}.

\begin{proposition} \label{prop:main_result_propagation}
    Let $\varphi \in \mathcal{D} \cap H^3 (\Omega)$ such that $v_0 \coloneqq \partial_x \varphi \in H^1_0$. Let $u \in \mathcal{C}(\mathbb{R}, H^2 (\Omega))$ solution to \eqref{logNLS} with Neumann boundary conditions and with $\varphi$ as initial data. Then there exists a maximal $T > 0$ such that $u \in \mathcal{C}((- T, T); H^3 (\Omega))$ and, if $T < \infty$, then $\inf_{t \in [- T, T], x \in \Omega} \abs{\mathcal{V} [\partial_x u (t)]} = 0$.
\end{proposition}

\begin{proof}
    Define
    \begin{equation*}
        T_1 \coloneqq \sup \{ T \mid \inf_{t \in (-T,T), x \in \Omega} \abs{\mathcal{V} [v (t)]} > 0 \}.
    \end{equation*}
    By applying Lemma \ref{lemma_odd} and Lemma \ref{lemma_minoration_dx_u}, we know that $T_1 > 0$. Moreover, if $T_1 < \infty$, then there obviously holds $\inf_{t \in [- T_1, T_1], x \in \Omega} \abs{\mathcal{V} [\partial_x u (t)]} = 0$
    
    Let $T \in (0, T_1)$, so that $\inf_{t \in (-T,T), x \in \Omega} \abs{\mathcal{V} [v (t]} > 0$.
    From Lemma \ref{lem:link_duhamel}, $v \coloneqq \partial_x u$ satisfies \eqref{eq:abstract_rel_for_Cauchy} with initial data $v_0 \coloneqq \partial_x \varphi$.
    As we know that $\partial_x u \in \mathcal{C}(\mathbb{R}, H^1_0 (\Omega))$, we can apply Corollary \ref{cor:Cauchy_less_abstract}, showing that $\partial_x u$ is the unique solution of \eqref{eq:abstract_rel_for_Cauchy} and that $\partial_x u \in \mathcal{C}([- T, T]; H^2 (\Omega))$, which means $u \in \mathcal{C}([- T, T]; H^3 (\Omega))$. Therefore, $u \in \mathcal{C}((- T_1, T_1); H^3 (\Omega))$.
\end{proof}

\section{Numerical experiments} \label{sec:numerics}

In this section, we illustrate our main results Theorem \ref{th:non-prop_regularity} and Theorem \ref{theo:propagation}, and extend the discussion to formulate new insights about the dynamics of equation \eqref{logNLS}. We use a standard spatial discretization of the 1D Laplace operator $\Delta_h$ with Neumann boundary conditions. The spatial domain $\Omega=[-16,16]$ is divided into $2^K+1$ spatial discretization points, where the integer $K$ is specified with each simulations. Note that the space step size is given by $h=\frac{2 \times 16}{2^K} = 2^{5-K}$.

We recall the definition of discrete Sobolev norms:
\[ \| u \|_{H^N_h}^2 \coloneqq \sum_{n=0}^N \| u \|_{\dot{H}^n_h}^2, \quad \| u \|_{\dot{H}^n_h}^2 \coloneqq \langle (-\Delta_h)^n u, u \rangle_h,  \]
where the discrete scalar product and discrete Laplace operator write
\[ \langle u, v \rangle_h = h \sum_{x \in h\Z \cap \Omega} u(x)\overline{v(x)} \quad \text{and} \quad \Delta_h u(x)= \frac{u(x+h)+u(x-h)-2u(x)}{h^2}   \]
for $x \in (h\Z \cap \Omega) \backslash \left\{-a,a \right\}$ and with associated Neumann boundary conditions. In the continuum limit $h \rightarrow 0$, these discrete norms converge to the standard Sobolev norms associated with $H^N(\Omega)$, ensuring consistency with the continuous case. 

To simulate the dynamics of \eqref{logNLS}, we use a second-order Strang splitting scheme. This approach involves splitting the equation into two sub-problems:
\[ \left|
\begin{aligned} 
& \ i  \partial_t v= - \Delta v, & \quad v(0)=v_0,  \\   
& \ i  \partial_t w= \lambda w \log|w|^2 , & \quad w(0)=w_0.
\end{aligned}
 \right. \]
For the logarithmic nonlinearity, we introduce the function $\phi$ defined as:
 \[ 
 \phi(w)=  \left\{
 \begin{aligned}
 &\log |w|^2 \quad \text{if} \ w \neq 0,  \\ 
 & 0 \quad \text{if} \ w = 0,
 \end{aligned} \right.
 \]
Note that although the expression $w \log |w|^2$ is mathematically well-defined at $w=0$, it leads to a numerical divergence at this point so we have to introduce the auxiliary function $\phi$ to avoid this pathological behavior. Solutions of the sub-problems are then given by the associated operators, for $t \in \R$,
\[  v(t)= \Phi^t_{\mathcal{L}} (v_0)= e^{it \Delta} v_0, \]
\[  w(t)=\Phi^t_{\mathcal{N}}(w_0)= e^{-i \lambda t \phi(w_0)} w_0. \]
The usual Strang splitting scheme recursively writes as 
\[  u^{j+1}= \Phi^{\frac{\tau}{2}}_{\mathcal{L}}  \circ \Phi^{\tau}_{\mathcal{N}} \circ \Phi^{\frac{\tau}{2}}_{\mathcal{L}} (u^j), \quad u^0=\varphi  \]
for all $0 \leq j \leq J$ with $J \tau = T$, where $\tau>0$ is the time step.

\subsection{Consistency of high discrete Sobolev norms} 
We fix a final time $T=0.01$ and $J=1000$. We take the initial condition
\[ \varphi(x)=\tanh(x), \quad x \in \Omega,  \]
so that $\varphi$ nearly satisfies Neumann boundary conditions. For increasing values of number of space discretization points $2^K+1$ with~$K=7,\ldots,11$, we plot the evolution in time of discrete Sobolev norms $H^N_h$ for several regularities $N=0,\ldots,5$ in Figure \ref{fig:evolution_discrete_Sobolev.png}.

 \begin{figure}[h!]
	\centering
    \hspace*{-2.5cm} % Shift the figure to the left
	\captionsetup{width=0.8\textwidth}
		\includegraphics[width=1.3\textwidth,trim = 0cm 0cm 0cm 0cm, clip]{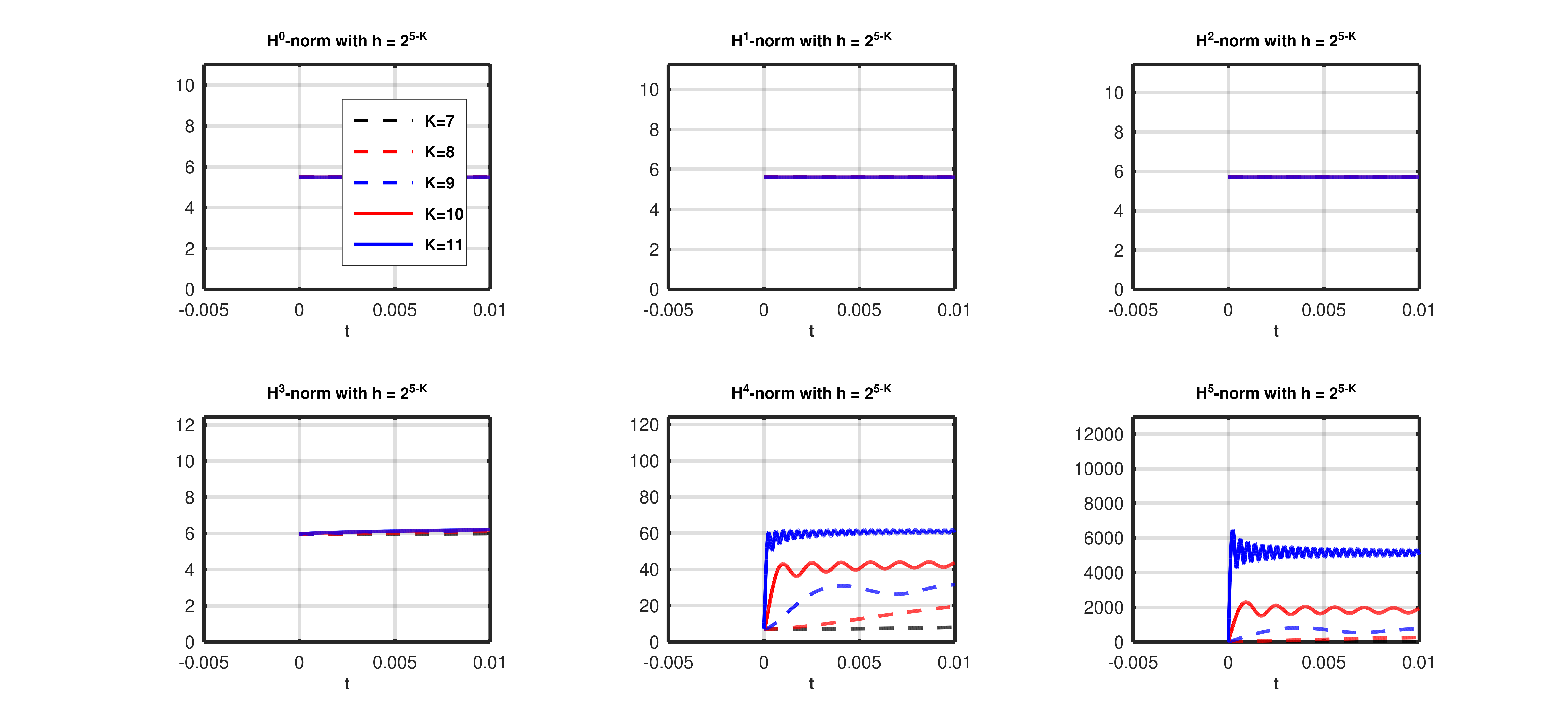}	
	\caption{Time evolution of the $H^s$-norms for the numerical solution $(u^j)_{0 \leq j \leq J}$ as the spatial discretization step $h$ tends to 0.}
	\label{fig:evolution_discrete_Sobolev.png}
\end{figure}

Let's first remark that although discrete Sobolev norms remain bounded due to their equivalence with the discrete $L^2_h$ norm, which is conserved under the Strang splitting flow, they are not uniformly bounded in $h$.

We observe that the discrete Sobolev norms up to $H^3$ converge to a finite limit as $h \rightarrow 0$. In contrast, the $H^4$- and the $H^5$-norms do not converge in the continuum limit and instead grow increasingly large as the spatial step size decreases. This aligns with the fact that for odd solutions non-vanishing outside 0, $H^3$ regularity is preserved for short times, while regularities higher than $\frac{7}{2}$ instantaneously blow-up. 

\subsection{Evolution of solutions} 

We now turn to an informal analysis of the dynamics of the logarithmic Schrödinger equation, focusing on the role of cancellation points. Using $K=8$, a final time $T=1$, and $J=1000$ time discretization points, we begin with the initial condition $\varphi(x)=\tanh(x)$. In Figure \ref{fig:evolution_tanh.png}, we plot the real and imaginary parts of the solution at several time steps. We observe an increasing number of oscillations over time caused by the cancellation point at the origin. However, the cancellation point still remains a first-order cancellation, so that our result should apply on the interval of simulation. This stress out the fact that our propagation of regularity result is not only a short time result but could in practice apply for longer times.

 \begin{figure}[ht]
	\centering
	\captionsetup{width=0.8\textwidth}
		\includegraphics[width=0.9\textwidth,trim = 20cm 5cm 20cm 10cm, clip]{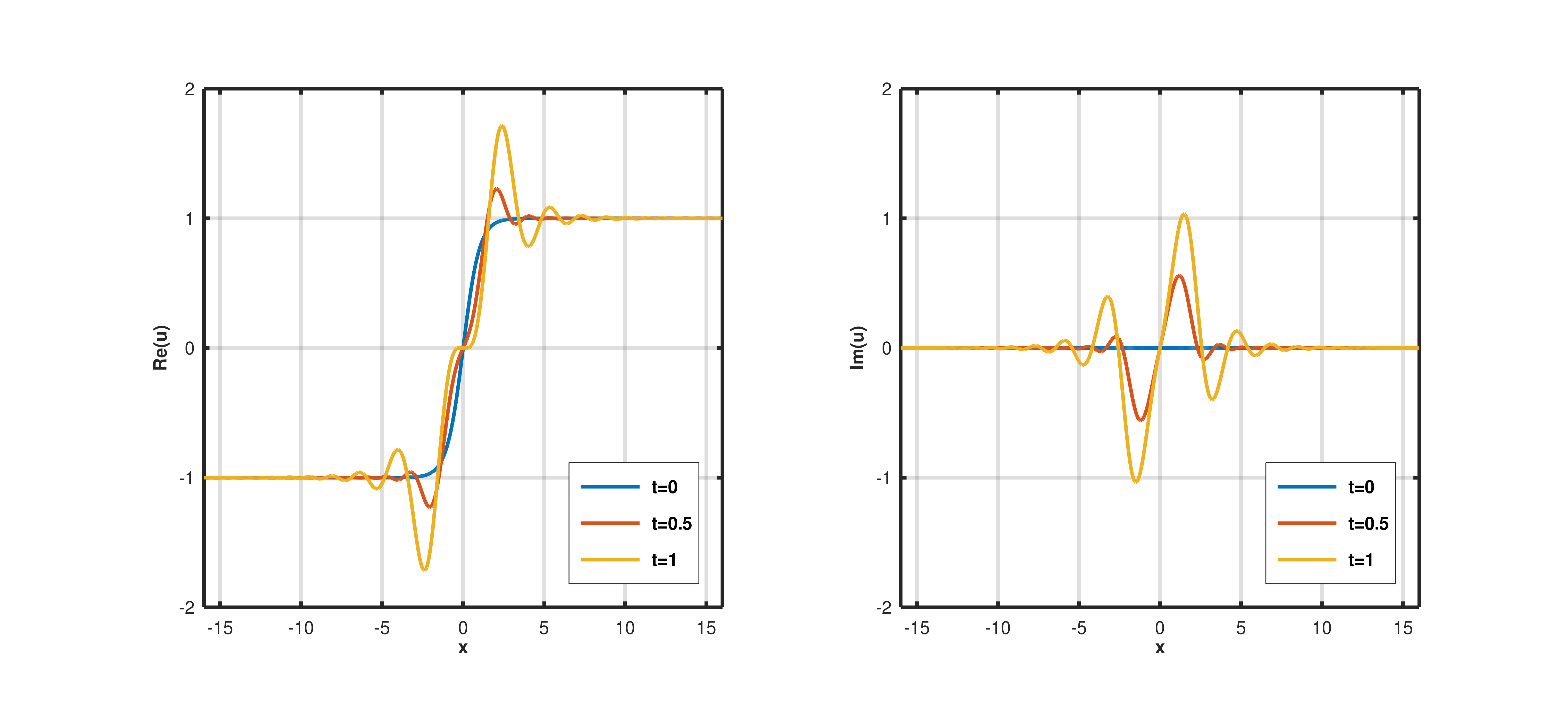}	
	\caption{Plots of the real part \textit{(left pannel)} and imaginary part \textit{(right pannel)} of solution of equation \eqref{logNLS} with initial condition $\varphi(x)=\tanh(x)$ at time $t=0$, $1$ and $2$.}
	\label{fig:evolution_tanh.png}
\end{figure}

Next, we consider the initial condition $\varphi(x)=1-\cos\left(\frac{\pi x}{16}\right)$ in Figure \ref{fig:evolution_cos.png} with parameters $K=11$, $T=0.1$ and $J=1000$. This setup introduces a second-order cancellation in the Taylor expansion around 0, a property that do not persist over time, as the solution immediately departs from zero. This behavior is illustrated in the right panel of Figure~\ref{fig:evolution_cos.png}, where we plot the minimum of the solution's absolute value. Regarding discrete Sobolev norms, similar simulations to those in the previous section indicate that the $H^3$-norm still appears to converge, while the $H^4$-norm does not.

\begin{figure}[ht]
	\centering
	\captionsetup{width=0.8\textwidth}
		\includegraphics[width=0.90\textwidth,trim = 15cm 5cm 10cm 10cm, clip]{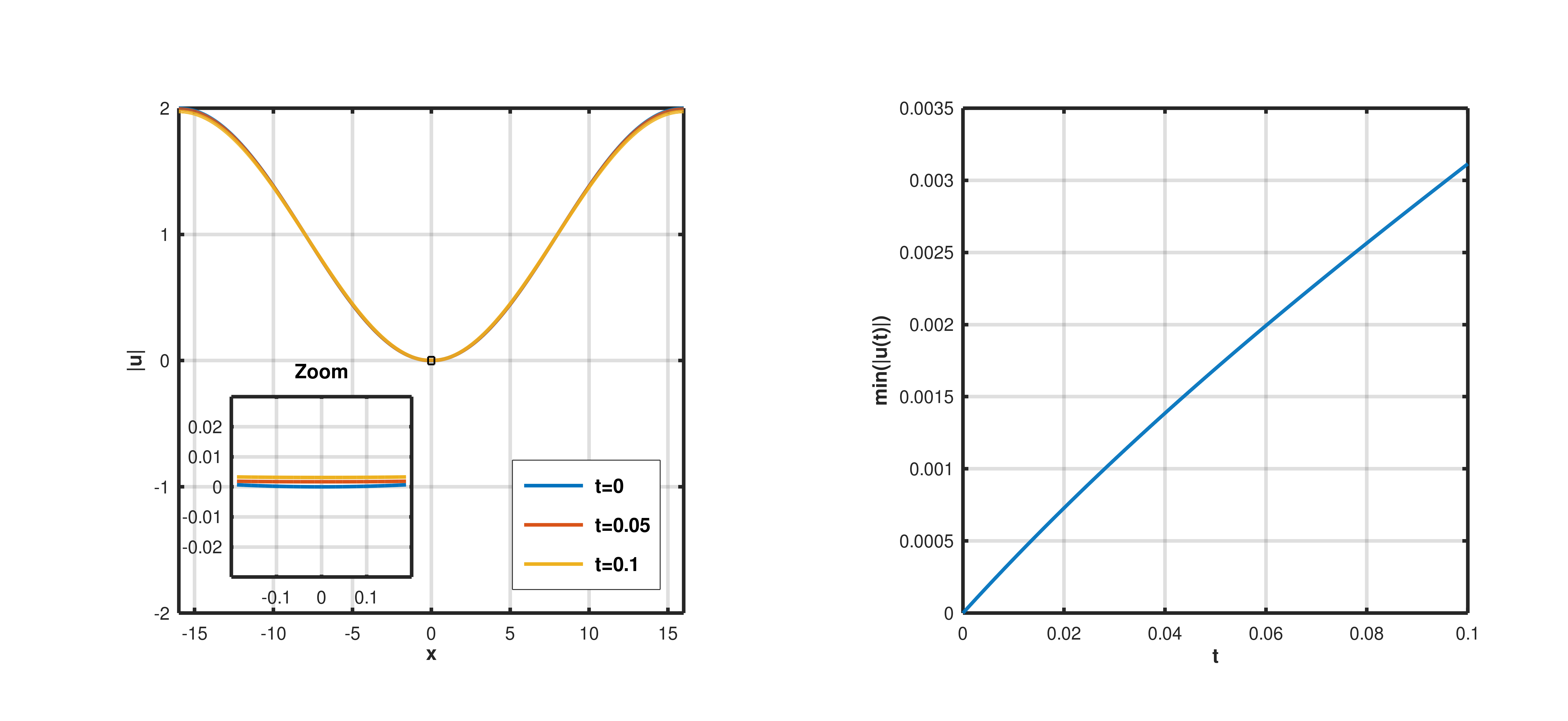}	
	\caption{Plots of the absolute value $|u|$ of solution of equation \eqref{logNLS} with initial condition $\varphi(x)=1-\cos\left(\frac{\pi x}{16}\right)$ at time $t=0$, $0.05$ and $0.1$ \textit{(left panel)}, as well as the evolution over time of the minimum of its absolute value \textit{(right panel)}.}
	\label{fig:evolution_cos.png}
\end{figure}

\section{Conclusion and perspectives} \label{sec:conclusion}

In this paper, we explored the role of cancellation points in the dynamics of the logarithmic Schrödinger equation. We demonstrated that a first-order cancellation leads to an instantaneous loss of regularity for $s>\frac{7}{2}$. Conversely, we proved that in the case of Neumann boundary conditions on a symmetric bounded one-dimensional domain, odd solutions with first-order cancellation and non-vanishing behavior outside $x=0$ preserve $H^3$ regularity. These theoretical results were further supported and enhanced by numerical simulations.\\

Of course several intriguing questions remain open, which we outline and discuss below. They also help to put the main results of this paper into perspective.

\begin{itemize}
\item Could we propagate further regularity $3<s \leq \frac{7}{2}$ for odd functions $\varphi \in \mathcal{D}\cap H^s(\Omega)$ exhibiting first-order cancellation? Addressing this question would require employing the fractional Sobolev-Slobodeckij seminorm
\[  \| f \|_{H^s}= \| f \|_{H^{\lfloor s \rfloor}}+ \int_{\Omega} \int_{\Omega} \frac{|f(x) - f(y)|^2}{|x-y|^{1+2(s-\lfloor s \rfloor)}} \dd x \dd y  \]
along with a careful adaptation of the equivalence properties of the operator $A$ discussed in Section~\ref{sec:_toy_model}, which is not straightforward.

\item On the other hand, is it possible to consider more general cancellation mechanisms for odd solutions of equation~\eqref{logNLS}, particularly ones leading to an instantaneous blow-up of the $H^s$-norm for $2<s\leq\frac{7}{2}$?
This problem seems quite hard, as the behavior of the solution is less clear: if $0$ is not a first-order cancellation point, other cancellation points may appear instantaneously near $0$; $0$ may remain a higher-order cancellation point without any other cancellation point to appear; or~$0$ may instantaneously become a first-order cancellation.
Describing precisely the nonlinearity and its effect on the behavior of the solution in each of these cases seems very complicated. 
In the latter case, which might be the more general case, a first step would be to find some "typical" even $v \in H^{s-1}((0,1))$ (say $v > 0$ for~$x > 0$) for~$s > 2$ such that the quantity $v \log \left(\int_0^x v(y) \dd y \right)$ is as less regular as possible, which is a nontrivial task.

 \item Could one also treat the case of even solutions, particularly those exhibiting second-order cancellations? Our numerical simulations from Section \ref{sec:numerics} suggest that such cancellations vanish instantaneously, which may lead to a propagation of some high regularity. Indeed, a quick computation shows that for a second-order cancellation~(i.e. with $\partial_{x}^2 \varphi (0) \neq 0$), we formally have $\partial_t u (0, 0) \neq 0$ so second-order cancellation is not propagated, in agreement with our numerical observations. A rigorous proof of such a result remains to be established.

 \item A very natural question is whether one can go beyond symmetric solutions of~\eqref{logNLS}, for instance by considering cases with parametrizable cancellation points. Naturally, such a parametrization introduces new terms in the non-moving frame, arising from time differentiation in the equation, which appear difficult to manage in a rigorous analysis.
 
 \item The behavior induced by cancellation points on other geometries, such as the torus~$\T$ or the full space~$\R$, also need to be understood. We strongly believe that our strategy should naturally extend to the periodic case (with double cancellation points) and to the full space for the logarithmic Gross-Pitaevskii equation, provided the non-vanishing behavior at infinity is carefully handled. By contrast,  the case of the logarithmic Schrödinger equation on $\R$ seems to remain a challenging problem, since cancellation points for large $x$ can instantaneously appear.
 
 \item Finally, it would be interesting to explore the case of higher dimensions in space. We emphasize that handling cancellation points in higher-dimensional bounded domains becomes significantly more challenging compared to the one-dimensional setting. 
 Indeed, in higher dimensions, a first-order cancellation point may not be isolated.
 % For example, the classical single-point cancellation associated with nonlinear Schrödinger equations, given by vortex profiles like $\frac{x}{|x|}$, lacks sufficient regularity to be treated with our current strategy. 
 Nevertheless, investigating the nonstandard behavior of solutions observed in the recent numerical study \cite{BaoMaWang2024} would be a very interesting direction for future work.
\end{itemize}

These questions point to several interesting directions for future research, and we plan to address some of these challenges in upcoming works.

\section*{Acknowledgments}
The authors acknowledges the support of the CDP C2EMPI, together with the French State under the France-2030 programme, the University of Lille, the Initiative of Excellence of the University of Lille, the European Metropolis of Lille for their funding and support of the R-CDP-24-004-C2EMPI project. The authors are grateful to Rémi Carles for his thoughtful comments on the final version of the manuscript.

\bibliographystyle{siam}
\bibliography{biblio.bib}

\end{document}